\documentclass{tran-l}

\usepackage{amssymb}

\usepackage{graphicx}

\newtheorem{theorem}{Theorem}[section]
\newtheorem{lemma}[theorem]{Lemma}

\newtheorem{proposition}[theorem]{Proposition}
\newtheorem{corollary}[theorem]{Corollary}

\theoremstyle{definition}
\newtheorem{definition}[theorem]{Definition}
\newtheorem{example}[theorem]{Example}

\theoremstyle{remark}
\newtheorem{remark}[theorem]{Remark}

\numberwithin{equation}{section}

\newcommand{\N}{\mathbb{N}}

\newcommand{\A}{{\mathcal A}}
\newcommand{\B}{{\mathcal B}}
\newcommand{\D}{{\mathcal D}}

\newcommand{\G}{{\mathcal G}}

\newcommand{\T}{\mathbf{T}}
\renewcommand{\t}{\mathbf{t}}

\renewcommand{\a}{\mathbf{a}}
\newcommand{\e}{\mathbf{e}}
\renewcommand{\u}{\mathbf{u}}
\renewcommand{\v}{\mathbf{v}}

\renewcommand{\P}{\mathbb{P}}
\newcommand{\CP}{\mathbb{CP}}
\newcommand{\C}{\mathbb{C}}
\renewcommand{\H}{\mathcal{H}}
\newcommand{\R}{\mathbb{R}}
\newcommand{\Hilb}{\mathrm{Hilb\,}}
\renewcommand{\span}{\mathrm{span}}
\newcommand{\supp}{\mathrm{supp}}

\newcommand{\dx}{\partial / \partial x}
\newcommand{\Z}{\mathbb{Z}}
\newcommand{\ZZ}{\mathcal{Z}}
\newcommand{\Dir}{\mathsf{Dir}}
\newcommand{\x}{\mathsf{x}}
\newcommand{\y}{\mathsf{y}}

\def\<{\left<}
\def\>{\right>}
\def\dd#1{{\partial}/{\partial x_{#1}}}
\def\<{\left<}
\def\>{\right>}
\def\ll{\langle\kern-3pt\langle}
\def\rr{\rangle\kern-3pt\rangle}

\begin{document}

\title{Combinatorics and geometry of power ideals.}


\author{Federico Ardila}
\address{Department of Mathematics, San Francisco State University, 1600 Holloway Ave, San Francisco, CA 94110, USA. }
\curraddr{}
\email{federico@math.sfsu.edu}
\thanks{Supported in part by NSF Award DMS-0801075.}

\author{Alexander Postnikov}
\address{Department of Mathematics, Massachusetts Institute of Technology, 77 Massachusetts Ave, Cambridge, MA 02139, USA.}
\curraddr{}
\email{apost@math.mit.edu}
\thanks{Supported in part by NSF CAREER Award DMS-0504629.}

\subjclass[2000]{05A15; 05B35; 13P99; 41A15; 52C35}

\date{}

\dedicatory{}

\begin{abstract}

We investigate ideals in a polynomial ring which are generated by powers of linear forms. Such ideals are closely related to the theories of fat point ideals, Cox rings, and box splines.

We pay special attention to a family of power ideals that arises naturally from a hyperplane arrangement $\A$. We prove that their Hilbert series are determined by the combinatorics of $\A$, and can be computed from its Tutte polynomial. We also obtain formulas for the Hilbert series of certain closely related fat point ideals and zonotopal Cox rings.

Our work unifies and generalizes results due to Dahmen-Micchelli, Holtz-Ron, Postnikov-Shapiro-Shapiro, and Sturmfels-Xu, among others. It also settles a conjecture of Holtz-Ron on the spline interpolation of functions on the lattice points of a zonotope.
\end{abstract}

\maketitle

\section{Introduction}

A \emph{power ideal} is an ideal $I$ in the polynomial ring $\C[V]$ generated by a collection of powers of homogeneous linear forms. One can regard the polynomials in $I$ as differential equations; the space of solutions $C$ of the resulting system is called a \emph{power inverse system}. We are particularly interested in a family of power ideals and power inverse systems which arise naturally from a hyperplane arrangement.

Such ideals arise naturally in several different settings. The following are some motivating examples:

\begin{itemize} 
\item (Postnikov-Shapiro-Shapiro \cite{PSS})
The flag manifold $Fl_n = SL(n,\C)/B$ has a flag of tautological vector bundles $E_0 \subset E_1 \subset \cdots \subset E_n$ and associated line bundles $L_i = E_i/E_{i-1}$. Let $w_i$ be the two-dimensional Chern form of $L_i$ in $Fl_n$. The ring generated by the forms $w_1, \ldots, w_n$ is isomorphic to 
\[
\Z[x_1, \ldots, x_n] 
\left/ 
\left< 
(x_{i_1} + \cdots + x_{i_k})^{k(n-k)+1} \, : \, 1 \leq i_1 < \cdots < i_k \leq n
\right> 
\right..
\]
Its dimension equals the number of forests on the set $[n]=\{1,\ldots, n\}$ and its Hilbert series enumerates these forests by number of inversions. The ideal above is one of the power ideals associated to the braid arrangement.

\item (Dahmen-Micchelli \cite{DM}, De Concini-Procesi \cite{DP}, Holtz-Ron \cite{HR})
Given a finite set $X=\{a_1, \ldots, a_n\}$ of vectors spanning $\R^d$, let the zonotope $Z(X)$ be the Minkowski sum of these vectors. The \emph{box spline} $B_X$ is a piecewise polynomial function on the zonotope $Z(X)$, defined as the convolution product of the uniform measures on the line segments from $0$ to each $a_i$. The box spline can be described combinatorially as a finite sum of local pieces. These local pieces, together with their derivatives, span a finite dimensional space of polynomials $D(X)$ which is one of the central objects in box spline theory. The space $D(X)$ is one of the power inverse systems associated to a hyperplane arrangement; its dimension is equal to the number of bases of $\R^d$ contained in $X$. Additionally, there are an \emph{external} and an \emph{internal} variant of the space $D(X)$ which also fit within this framework.

\item (Emsalem-Iarrobino \cite{EI}; Geramita-Schenck \cite{GS}) 
Given points $p_1, \ldots, p_n$ in projective space and positive integers $o_1, \ldots, o_n$, the corresponding \emph{fat point ideal} is the ideal of polynomials which vanish at each $p_i$ to order $o_i$. The Hilbert series of a fat points ideal can be expressed in terms of Hilbert series of power ideals.

\item (Sturmfels-Xu, \cite{SX})
A finite set of points $\{p_1, \ldots, p_n\}$ in $\P^{d-1}$ determines a Cox-Nagata ring, which is a multigraded invariant ring of a polynomial ring.  It can be interpreted as the Cox ring of the variety obtained from $\P^{d-1}$ by blowing up $p_1, \ldots, p_n$. Nagata used such rings to settle Hilbert's 14th problem. The 
multigraded Hilbert series of a Cox-Nagata ring can be expressed in terms of the Hilbert series of a family of power inverse systems. Certain subrings of Cox rings, called zonotopal Cox rings, are intimately related to the power inverse systems of a hyperplane arrangement.

\item (Berget \cite{Be}, Brion-Verge \cite{BV}, Orlik-Terao \cite{OT}, Proudfoot-Speyer \cite{PrS}, Terao \cite{Te})
Given a hyperplane arrangement determined by the linear functionals $\alpha_1, \ldots, \alpha_n$, various subalgebras of the algebra generated by $\frac1{\alpha_1}, \ldots, \frac1{\alpha_n}$ have been studied, in some cases with additional structure. Some of these algebras are related to the objects in this paper, as outlined in \cite{Be}.
\end{itemize}

The paper is organized as follows. In Section \ref{sec:ideals} we discuss general power ideals $I(\rho)$ and the corresponding inverse systems $C(\rho)$, and associate a projective variety to each power ideal. In Section \ref{sec:poly} we associate a power ideal $I(\rho_f)$ to each homogeneous polynomial $f(\x)$, whose associated variety is the hypersurface $f(\x)=0$. We show that the smoothness of the hypersurface is detected by the Hilbert series of $C(\rho_f)$. Section \ref{sec:hyparr} is devoted to the special case that most interests us: the family of power ideals $I_{\A,k}$ and inverse systems $C_{\A,k}$ associated to a hyperplane arrangement. We compute the Hilbert series of the spaces $C_{\A,k}$ in terms of the combinatorics of $\A$, and find explicit bases for them. These computations and constructions simultaneously generalize numerous results in the literature, and prove a conjecture of Holtz and Ron about these spaces. Section \ref{sec:fatpoints} applies the results of Section \ref{sec:hyparr} to compute the Hilbert series of a family of fat point ideals which one can naturally associate to $\A$. Section \ref{sec:zono} then applies these results to give an explicit formula for the multigraded Hilbert series of the zonotopal Cox ring of $\A$. We conclude with some open questions.

\section{Power ideals and inverse systems}\label{sec:ideals}
\subsection{Power ideals}
Let $V\simeq \C^n$ be a finite-dimensional vector space over $\C$ and $V^*$ the dual space. 

\begin{definition}
A \emph{power ideal} is an ideal in the polynomial ring 
$\C[V]$ generated by a collection of powers of homogeneous linear 
forms such that these linear forms span $V$; \emph{i.e.}, an ideal of the form 
$\left< h_i^{r_i}: i \in I\right>$ where $I$ is some indexing set, the $h_i$s are linear forms which span $V$, and the $r_i$s are non-negative integers.
%
\end{definition}

Since the linear forms $h$ span the space $V$, the algebra $A=\C[V]/I$ has finite dimension
$\dim A > 0$. The ideal $I$ is homogeneous so the algebra $A$ is graded:
$A=A_0\oplus A_1 \oplus A_2 \oplus \cdots$.  In this paper we 
calculate the dimension of $A$ and its Hilbert series
$\Hilb A = \sum_{i\geq 0} \dim A_i \, q^i$ for an important family of power ideals $I$.

\begin{example}  Let $I = \left< x_1^4, x_2^2, (x_1+x_2)^3\right>\subset \C[x_1, x_2]$.
The algebra $A=\C[x_1,x_2]/I$ has the basis $1,x_1,x_2,x_1^2,x_1x_2, x_1^3$, so
$\Hilb (A;q) = 1 + 2q + 2q^2 + q^3$.
\end{example}

Let $\rho: \P V\to \N$ be a nonnegative integer function on
the projective space $\P V \simeq \CP^{n-1}$.
We will identify $\rho$ with a function on $V\setminus\{0\}$
such that $\rho(t\cdot a) = \rho(a)$ for $t\in\C\setminus\{0\}$.  
Let $I(\rho)\in \C[V]$ be the power ideal generated by the powers of linear
forms $h^{\rho(h)+1}$ for all $h\in V\setminus \{0\}$.

Any power ideal $I$ is of the form $I(\rho)$.  
Indeed, for any linear form $h\in V$, there is some positive integer $r$ such that $h^r \in I$.  
Let $\rho_I:\P V\to \N$ be the function such that $\rho_I(h)+1$ equals
the {\it minimum\/} integer $r$ such that $h^r \in I$.  This is clearly the unique minimum function $\rho: \P V \to \N$ such that $I = I(\rho)$.

\begin{definition}
Given a power ideal $I$, let the \emph{directional degree function} of $I$ be $\rho_I$, the minimum function from $\P V \to \N$ such that $I = I(\rho_I)$. 
\end{definition}

The name we give to these functions is justified by Proposition \ref{prop:dirdegree}.

Let $\Dir(V)$ be the set of directional degree functions on $\P V$.
Then power ideals $I$ in $\C[V]$ are in bijection with $\Dir(V)$.
For any nonnegative integer function $\rho$ on $\P V$ there is a directional degree function $\rho'\in \P V$
such that $I(\rho)=I(\rho')$.

We say that a set of points $h_1,\dots,h_N\in V\setminus\{0\}$ is a
{\it generating set\/} for a power ideal $I(\rho)$
if $I(\rho) = \< h_1^{\rho(h_1)+1},\ldots,h_N^{\rho(h_N)+1}\>$. Hilbert's basis theorem guarantees the existence of such a set.

These concepts raise several natural questions, which we do not address here.

\medskip
\noindent
{\bf Questions.}  Find a nice description of the space $\Dir(V)$ of directional degree functions on $\P V$? 
Find an efficient way of computing the directional degree function $\rho_I$ associated to a given power ideal $I = \<h_1^{r_1},\dots, h_N^{r_N}\>$ or, more generally, to an arbitrary non-negative integer function on $\P V$. Find a generating collection of points for a power ideal $I(\rho)$? 
\medskip

\subsection{Inverse systems}

There is a very useful dual way of thinking about power ideals in terms of Macaulay inverse systems, which we now outline.

\begin{definition}
A \emph{Macaulay inverse system} (or simply \emph{inverse system}) is a finite dimensional space of polynomials which is closed under differentiation with respect to the variables.
\end{definition}

First, we define a pairing $\<\cdot,\cdot\>$ between the polynomial rings $\C[V]$ and $\C[V^*]$.
Let $x_1,\dots,x_n$ be a basis of $V$, and let $y_1,\dots,y_n$ be the dual basis of $V^*$.
For each $f(\x)=f(x_1,\dots,x_n) \in \C[V]$,
define a differential operator $f(\partial/\partial \y) := f(\partial/\partial y_1,\dots, \partial/\partial y_n)$ on $\C[V^*]$.
Similarly, for each $g(\y)=g(y_1,\dots,y_n)\in \C[V^*]$,
define a differential operator on $\C[V]$ by
$g(\partial/\partial \x) := g(\partial/\partial x_1,\dots, \partial/\partial x_n)$.
For $f \in \C[V]$ and $g \in \C[V^*]$, define
$$
\<f,g\> := \left.f\left(\frac{\partial}{\partial \y}\right)\cdot g(\y)\right|_{\y=0} =
\left.g\left(\frac{\partial}{\partial \x}\right)\cdot f(\x)\right|_{\x=0}.
$$

\begin{definition}
The \emph{inverse system of a homogeneous ideal} $I \in \C[V]$ is its orthogonal complement with respect to this pairing, which is easily seen to be
$$
I^\perp := \left\{g(\y)\in\C[V^*] \left| f\left(\frac{\partial}{\partial \y}\right)\cdot g(\y) = 0 \right.
\text{ for any } f(\x)\in I\right\}.
$$
The inverse system of the power ideal $I(\rho)$ is called the \emph{power inverse system} $C(\rho) = I(\rho)^{\perp}$.
\end{definition}
Since $I^\perp$ is the space of solutions of a system of homogeneous differential equations 
with constant coefficients, it is an inverse system.
The space $I^\perp$ is graded.
The dimension of the algebra $A= \C[V^*]/I$ equals $\dim I^\perp$. 
Moreover, the dimension of the $i$-th graded component $A_i$ of $A$ equals
the dimension of the $i$-th graded component $(I^\perp)_i$ of $I^\perp$.

\begin{proposition}\label{prop:dirdegree}
The power inverse system $C(\rho)$ consists of the polynomials $f(\y) \in \C[V^*]$ whose restriction to any affine line in direction $h \in V$ has degree at most $\rho(h)$.
\end{proposition}

\begin{proof}
A polynomial $f(\y)\in\C[V^*]$ satisfies $h(\partial/\partial \y)^{\rho(h)+1} f(\y) = 0$ for a linear form $h\in V\setminus\{0\}$ if and only if the restriction $r(t) := f(\x + t\,h) \in \C[t]$ of the polynomial $f$ to any affine line of the form $L= \{\x+t\,h\mid t\in\C\}\subset V$ is a polynomial in $t$ of degree at most $\rho(h)$. 
\end{proof}

\begin{definition}
Given a polynomial $f \in \C[V^*]$, let the \emph{directional degree function of $f$} be $\rho_f: \P V \to \N$, defined by letting $\rho_f(h)$ be the degree of the restriction of $f$ to a generic line parallel to $h$. For a set of polynomials $S\subset \C[V^*]$ with finite dimensional linear span, define the {\it degree-span\/} $\ll S\rr$ as the unique minimal space $C(\rho)$ such that $C(\rho) \supseteq S$. 
\end{definition}

A polynomial $f \in \C[V^*]$ belongs to the degree-span $\ll S\rr$ if and only if,
for any affine line $L$ in direction $h$, the degree of
$f$ along $L$ is less then or equal to the largest
degree of a polynomial $g \in S$ along a line parallel to $L$.
In symbols, $\rho(h) = \max_{g \in S} \rho_g(h)$ and 
\[
\ll S \rr = \left\{ f \in \C[V^*] \left| \rho_f(h) \leq \max_{g \in S} \rho_g(h)\right.\right\}
\]
In particular, $\ll f \rr = C(\rho_f)$.
Since the space $C(\rho)$ is finite-dimensional, there is a finite collection
of polynomials $f_1,\dots,f_N$ such that $C(\rho) = \ll f_1,\dots,f_N\rr$. 


The following propositions list some properties of the spaces $C(\rho)$ orthogonal to power ideals, and of the space of directional degree functions.

\begin{proposition} 
\label{prop:C_properties}
Power inverse systems have the following properties:

{\rm 1.}
If $\ll S\rr = C(\rho_1)$
and $\ll T\rr = C(\rho_2)$,
then $\ll S \cup T \rr = C(\max(\rho_1,\rho_2))$ 
and $\ll S \cdot T \rr = C(\rho_1+\rho_2)$.

\smallskip

{\rm 2.}
For any power inverse system $C(\rho)$, 
and any $f\in C(\rho)$ we have that:

\begin{enumerate}
\item[(a)] Any partial derivative $\partial f/\partial x_i$ belongs to $C(\rho)$.
\item[(b)] Any shift $f(x+x_0)$, for fixed $x_0\in \C^n$, belongs to $C(\rho)$.
\item[(c)] Any polynomial that divides $f$ belongs to $C(\rho)$.
\item[(d)] If $f = f_0+f_1 + \dots + f_d$, where $f_i$ is a homogeneous polynomial
of degree $i$, then all $f_i$ belong to $C(\rho)$.
\end{enumerate}

\end{proposition}

\begin{proof}

1.  The degree of the restriction of any element in $S\cup T$ to an
affine line $L$ in the direction $a$ is less than or equal
to $\max(\rho_1(a),\rho_2(a))$.
Thus $\ll S \cup T \rr = C_{\rho_3}$, for some $\rho_3 \leq \max(\rho_1,\rho_2)$.
On the other hand, there is an element in $S\cup T$ whose degree of restriction to $L$
is exactly $\max(\rho_1(a),\rho_2(a))$.  Thus $\rho_3 = \max(\rho_1,\rho_2)$.
A similar argument works for $\ll S\cdot T\rr$.
 
2.  Statements (a),(b),(c) are trivial from the description of the space $C(\rho)$ in terms of degrees
of restrictions to affine lines.  
Statement (d) follows from the fact that $C(\rho)$ is a graded space.
\end{proof}

\begin{proposition} 
\label{prop:rho_properties}
Directional degree functions have the following properties:

{\rm 1.}
For $\rho \in \Dir(V) $, and any $a,b,a+b\in \C^n\setminus\{0\}$,
we have the triangle inequality $\rho(a) + \rho(b) \geq \rho(a+b)$.
\smallskip

{\rm 2.}
For $\rho_1,\rho_2\in \Dir(V)$, the functions $\rho_1 + \rho_2$
and $\max(\rho_1,\rho_2)$ belong to $\Dir(V)$.
In other words, $\Dir(V)$ is closed under the operations
``$+$'' and ``$\max$''.

\smallskip
 
{\rm 3.}
For any $\rho \in \Dir(V)$ and polynomials
$f_1,\dots,f_N$ such that
$\ll f_1,\dots,f_N\rr = C(\rho)$,
we have $\rho = \max(\rho_{f_1},\dots,\rho_{f_N})$.

\end{proposition}

\begin{proof}
To prove the first statement, notice that since $a^{\rho(a)+1}$ and $b^{\rho(b)+1}$ are in $I(\rho)$, $(a+b)^{\rho(a)+\rho(b)+1}$ is also in $I(\rho)$ by the binomial theorem.
 The other two statements are immediate consequences of the first part of Proposition \ref{prop:C_properties}.
\end{proof}

According to Proposition~\ref{prop:rho_properties}.3, to describe all 
functions $\rho \in \Dir(V)$, it is enough to describe the functions $\rho_{f}$ for all homogeneous polynomials $f$. To do that, we will use the following lemma.

\begin{lemma}{\rm\cite{EI}}\label{lemma:degreeorder}
A homogeneous polynomial of degree $d$ in $\C[V^*]$ has degree $d'$ along direction $h$ if and only if it vanishes exactly to order $d - d'$ at $h$.
\end{lemma}

\begin{proof}
Let $f$ be the polynomial. Consider $a \in V$ and $t \in \C$. We have that 
$$
f(th+a) =  
\sum_{k_1,\dots,k_n \geq 0}
\left.
\frac{a_1^{k_1}}{k_1!}\cdots \frac{a_n^{k_n}}{k_n!}
\left(\frac{\partial}{\partial x_1}\right)^{k_1} \cdots 
\left(\frac{\partial}{\partial x_n}\right)^{k_n} f(x)\right|_{x=h} t^{d-(k_1+\cdots+k_n)}.
$$
The terms of $t$-degree greater than $d'$ cancel if and only if all the derivatives of $f$ of order less than $d-d'$ vanish at $h$.
\end{proof}

\begin{proposition}
\label{prop:flag_f} 
For each homogeneous polynomial $f \in \C[V^*]$ of degree $d$ there is a flag 
$\emptyset = X_{-1}\subset X_0 \subset X_1 \subset \cdots \subset X_d = \P V \simeq \CP^{n-1}$
of projective algebraic sets such that $\rho(a) = i$ for $a\in X_i\setminus X_{i-1}$.

The algebraic set $X_{i}$ is the set of common zeros of 
$(\dd 1)^{k_1} \cdots (\dd n)^{k_n}f(x)$ for $k_1+\cdots + k_n\leq d-i-1$.
%
\end{proposition}

\begin{proof}
This is a straightforward consequence of Lemma \ref{lemma:degreeorder}.
\end{proof}

%

\begin{proposition}
\label{prop:flag_any_rho}
For each directional degree function $\rho\in \Dir(V)$ there is a flag of projective algebraic sets 
$\emptyset = X_{-1}\subset X_0 \subset X_1 \subset \cdots \subset X_d = \P V \simeq \CP^{n-1}$
such that $\rho(a) = i$ for $a\in X_i\setminus X_{i-1}$.

If $C(\rho) =\ll f_1,\dots,f_N\rr$ where $f_j$
is homogeneous of degree $d_j$ for $1 \leq j \leq N$,
then $d=\max(d_1,\dots,d_N)$ and $X_{i}$ is the set of common zeros
of all derivatives of the form $(\dd 1)^{k_1} \cdots (\dd n)^{k_n}f_j(x)$ with $1 \leq j \leq N$ and $k_1 + \cdots + k_n \leq d_j - i -1$.

In particular, if $d=d_1=\dots = d_M>d_{M+1}\geq \dots\geq d_N$,
then $X_{d-1}$ is the set of common zeros of the polynomials $f_1, \dots, f_M$.
\end{proposition}

\begin{proof}
According to Proposition~\ref{prop:rho_properties}, we have
$\rho = \max(\rho_{f_1},\dots,\rho_{f_N})$.
Suppose that the polynomial $f_i$ produces the flag 
$X_{-1}^i\subset X_0^i \subset X_1^i \subset X_2^i \subset \cdots$,
as in Proposition~\ref{prop:flag_f}.
Then $\rho$ corresponds to the flag $X_{-1}\subset X_0 \subset X_1 \subset X_2 \subset \cdots$,
where $X_j = X_j^1\cap \cdots \cap X_j^N$.
\end{proof}

Clearly, $\rho(h) = d$ for a generic point $h\in \P V$.
Define the {\it characteristic variety\/} $X=X(\rho)$ of the 
power ideal $I(\rho)$ 
as the locus of points $h\in\P V$ where $\rho(h)<d$.
Any projective variety is the characteristic variety of some power ideal.
%

\begin{remark}
Proposition~\ref{prop:flag_any_rho} shows 
that the structure of an arbitrary power ideal $I_\rho$ 
is at least as complicated as the structure of an arbitrary projective variety.
\end{remark}

\section{The power ideal of a homogeneous polynomial}\label{sec:poly}

Let $f$ be a homogeneous polynomial $f \in \C[V^*]$ and let $X=\{\x \in V \, | \, f(\x) = 0\}$ be the corresponding hypersurface in $V$. We defined the directional degree function $\rho_f:\P V \to \N$ of $f$ by letting $\rho_f(h)$ be the degree of $f$ on a generic line in direction $h \in V$. To $\rho_f$ we also associate a power ideal $I(\rho_f)$ whose characteristic variety is $X$. 

More generally, let $f_1, \ldots, f_N$ be degree $d$ polynomials in $\C[V^*]$ and consider the algebraic set $X=\{\x \in V \, | \, f_i(\x) = 0 \textrm{ for } 1 \leq i \leq N\}$. The directional degree function $\rho(h) = \max_{1 \leq i \leq N} \rho_{f_i}(h)$ defines a power ideal $I(\rho_{f_1}, \dots, \rho_{f_N})$ whose characteristic variety is $X$.
%

The following result tells us that $C(\rho_f)$ can detect the smoothness of the hypersurface $f(\x)=0$.

\begin{proposition}\label{prop:smooth}
Let $f\in\C[V^*]$ be a homogeneous polynomial of degree $d$, and let 
$X = \{\x\, |\, f(\x)=0\}\subset \P V \simeq \CP^{n-1}$ be the corresponding hypersurface.
The Hilbert series of the inverse system $C(\rho_f)$ is
$$
\Hilb (C(\rho_f);q) 
= \left(\sum_{i=0}^{d-1} {n+i-1\choose i} q^i\right) + q^d.
$$
if and only if $X$ is smooth.
\end{proposition}

\begin{proof}
First assume that $X$ is smooth. By Lemma \ref{lemma:degreeorder}, $\rho_f(x)$ is equal to $d-1$ for $x \in X$ and is equal to $d$ elsewhere. The polynomials $g \in C(\rho_f)$ are those whose restrictions to lines $X$ have degree at most $d-1$ and whose restrictions to other lines have degree at most $d$. Any polynomial of degree $d-1$ satisfies these conditions, and no polynomial of degree greater than $d$ satisfies them. A polynomial $g$ of degree $d$ which satisfies them must vanish at $X$, using 
Lemma \ref{lemma:degreeorder} again; therefore it must be a constant multiple of $f$. The desired result follows.

Now assume that $X$ is not smooth. Then $f$ vanishes at some point $h$ to order at least $2$, and hence has degree at most $d-2$ along that direction. It follows that $h^{d-1}$ is not in $C(\rho_f)$, which means that $ \dim C(\rho_f)_{d-1} < {n+d-2 \choose d-1}$.
\end{proof}

We now investigate the power ideal of a homogeneous polynomial in two cases: elliptic curves and hyperplane arrangements.

\subsection{A case study: Elliptic curves}

In this section we consider the power ideals determined by curves in the projective plane $\CP^2$ defined by 
an equation $f(x_1,x_2,x_3)= x_1^3 + ax_1 x_3^2 + bx_3^3 - x_2^2x_3=0$ where $a, b\in \C$ are two fixed constants. 
Such a curve $X$ can be parametrized as 
$$
X=\{(t:\pm r(t):1)\mid t\in\C\}\cup \{(0:1:0)\},
$$
where $r(t) := \sqrt{t^3+at+b}$.

%

The characteristic variety of the power ideal $I(\rho_f)\in \C[x_1,x_2,x_3]$ is the curve $X$.
To describe this power ideal $I(\rho_f)$, we need to consider three cases, shown in Figure \ref{fig:elliptic} in the real case. Generically, $X$ is non-singular, and it is called an elliptic curve. When $(a/3)^3 + (b/2)^2 = 0$ it has a double root, and when $a=b=0$ it has a cusp. 

\begin{figure}[htbp]
\begin{center}
\includegraphics[width=3.5in]{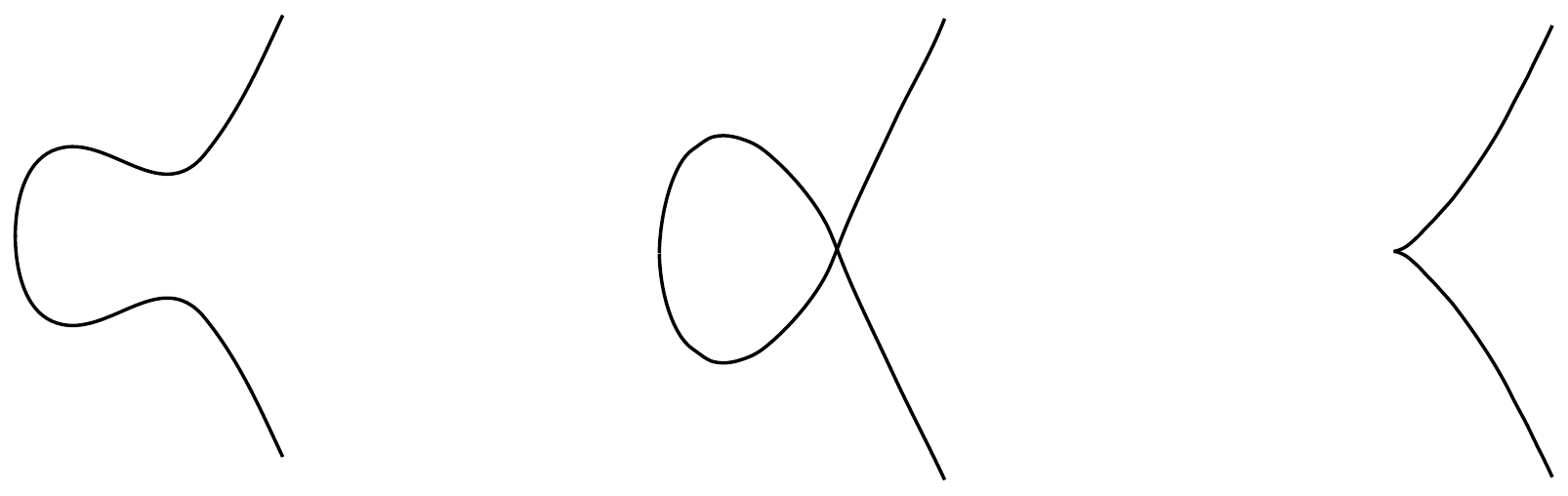} 
\caption{The three possibilities for the curve $x_1^3 + ax_1 x_3^2 + bx_3^3 - x_2^2x_3=0$: an elliptic curve, a curve with a double root, and a curve with a cusp.}
\label{fig:elliptic}
\end{center}
\end{figure}

\medskip
(I)  $(a/3)^3 + (b/2)^2 \ne 0$.  In this case the elliptic curve $X$ has no singular points, and 
we have
$$
\rho_f(h) = \left\{
\begin{array}{cl}
3 &\textrm{if } h\not\in X,\\
2 &\textrm{if } h\in X.\\
\end{array}\right.
$$

The ideal $I(\rho_f)\in\C[x_1,x_2,x_3]$ is generated 
by the powers $(v_1 x_1+v_2x_2+v_3x_3)^3$ for $(v_1:v_2:v_3)\in X$, and by all monomials $x_1^ix_2^jx_3^k$ of degree 4.
In other words, $I(\rho_f)$ is generated by $x_2^3$, $(tx_1\pm r(t) x_2 + x_3)^3$, for any $t\in \C$, 
and by all monomials of degree 4.
We have
\begin{eqnarray*}
& &(tx_1\pm r(t) x_2 + x_3)^3 \\
&=&(tx_1+x_3)^3 + 3r(t)^2(t x_1 + x_3)x_2^2 \pm 3r(t)(tx_1+x_3)^2x_2 
\pm r(t)^3 x_2^3 \\
&=&(x_3^3 + 3bx_2^2 x_3) + 3(x_1x_3^2 + b x_1x_2^2 + a x_2^2 x_3)t +
3(x_1^2 x_3 + ax_1x_2^2)t^2\\
& &  + (x_1^3 + 3x_2^2x_3)t^3 + 3x_1x_2^2t^4
\pm (3x_2x_3^2 + bx_2^3)r(t)
\pm (6x_1x_2x_3+ax_2^3)r(t)t\\
& &\pm 3x_1^2x_2r(t)t^2
\pm x_2^3r(t)t^3
%
\end{eqnarray*}

Since $1,t,t^2,t^3,t^4,r(t),r(t)t,r(t)t^2,$ and $r(t)t^3$ are linearly independent functions,

\begin{eqnarray*}
I(\rho_f) &=&\left<x_3^3 + 3bx_2^2 x_3, \ x_1x_3^2 + b x_1x_2^2 + a x_2^2 x_3, \
x_1^2 x_3 + ax_1x_2^2, \ (x_1^3 + 3x_2^2x_3)\right.\\
& & \left. x_1x_2^2, \ (3x_2x_3^2 + bx_2^3),\ 
(6x_1x_2x_3+ax_2^3), \ x_1^2x_2, x_2^3,  \ x_1^ix_2^jx_3^k\mid i+j+k=4\right> \\
&=&\left<
x_2^3,\ 
x_1x_2^2,\ 
x_2x_3^2,\ 
x_1x_2x_3,\ 
x_1^2x_2,\
x_3^3 + 3bx_2^2 x_3,\
x_1^2 x_3,\ 
\right.\\
& &
\left. 
x_1x_3^2 + a x_2^2 x_3,\
x_1^3 + 3x_2^2x_3,\
 x_1^ix_2^jx_3^k\mid i+j+k=4\right>.
\end{eqnarray*}


The space $C(\rho_f)$ is spanned by all polynomials in $\C[y_1,y_2,y_3]$ of degree at most $2$
and by the polynomial $f(y_1,y_2,y_3)$.
Thus $\Hilb(C(\rho_f); q) = 1 + 3q + 6q^2 + q^3$. This agrees with Proposition \ref{prop:smooth} since $X$ is smooth in this case.

\medskip
(II)  $(a/3)^3 + (b/2)^2 = 0$ and $a,b\ne 0$.
In this case the curve $X$ has one singular point, which is an ordinary double point:
$p_s=(-\frac{3b}{2a}:0:1)$. 
We have
$$
\rho_f(h) = \left\{
\begin{array}{cl}
3 &\textrm{if } h\not\in X,\\
2 &\textrm{if } h\in X\setminus\{p_s\},\\
1 &\textrm{if } h =p_s.
\end{array}\right.
$$
The ideal $I(\rho_f)$ is generated by all generators from case (I) and 
by $(-\frac{3b}{2a}x_1 + x_3)^2$, so we have
$\Hilb(C(\rho_f); q)= 1 + 3q + 5q^2 +q^3$.

\medskip
(III) $a=b=0$.
In this case the curve $X$ has one singular point, which is a cusp:
$p_c=(0:0:1)$.  We have
$$
\rho_f(h) = \left\{
\begin{array}{cl}
3 &\textrm{if } h\not\in X,\\
2 &\textrm{if } h\in X\setminus\{p_c\},\\
1 &\textrm{if } h =p_c.
\end{array}\right.
$$
The ideal $I(\rho_f)$ is generated by all generators from part (I) and 
by $x_3^2$.  Therefore
in this case we have $\Hilb(C(\rho_f); q)= 1 + 3q+5q^2+q^3$ also.

\begin{remark}  These examples show that, while the Hilbert series $\Hilb (\C[V^*]/I(\rho_f); q)$ determines whether the hypersurface $f(\x)=0$ is smooth, it may not distinguish between different types of singularities.
\end{remark}

\subsection{Hyperplane arrangements}\label{subsec:hyparr}

Consider the case where $f$ is a product of linear forms; say $f=l_1\ldots l_n$ where $l_1, \ldots, l_n \in V^*$. These forms define a hyperplane arrangement $\A = \{H_1, \ldots, H_n\}$ in $V$, and the hypersurface $X$ is the union of these hyperplanes.

\begin{proposition}
The directional degree function associated to a product of linear forms $f=l_1\ldots l_n$ is given by
\[
\rho_f(h) = \textrm{ number of hyperplanes in $\A$ not containing $h$.}
\]
\end{proposition}

\begin{proof}
Along a line in direction $h \in V$, we have
\[ 
f(a+th) = \prod_{i=1}^n l_i(a+th) = \prod_{i=1}^n \left(l_i(a) + tl_i(h)\right).
\]
It follows that the $t$-degree of $f$ along this line is equal to the number of $l_i$s which don't vanish at $h$, as desired.
\end{proof}

For reasons which will soon become clear, we will study the power ideals determined by the functions $\rho_f(h)+k$ for $k \in \Z, k \geq -2$.
These power ideals that arise from hyperplane arrangements have many interesting properties. In particular, their Hilbert series only depend on the combinatorial structure of the arrangements, and can be computed explicitly in terms of the Tutte polynomial \cite{Ar07} of the arrangement. Section \ref{sec:hyparr} is devoted to this important case.

\section{Power ideals of hyperplane arrangements}\label{sec:hyparr}

In this section we focus on the interesting family of ideals related to a hyperplane arrangement $\A$ which arises from the previous construction. We will see that the Hilbert series of these ideals depend only on the matroid $M(\A)$, which stores the combinatorial structure of $\A$. We will need some basic facts about matroids, Tutte polynomials, and their connection with hyperplane arrangements. We will outline the necessary background information, and we refer the reader to \cite{Ar07, Ox,St} for further details.

\subsection{The ideals $I_{\A, k}$ and $I'_{\A, k}$.}

Let $\A$ be a central hyperplane arrangement in $V$; that is, a finite collection of hyperplanes $H_1, \ldots, H_n$, where $H_i = \{x \mid l_i(x)=0\}$ for some linear functional $l_i \in V^*$. We can also think of $\A$ as the vector arrangement $\{l_1, \ldots, l_n\}$ in $V^*$. Let $M(\A)$ be the matroid of $\A$.

Each hyperplane $H_i$ has a corresponding directional degree function $\rho_{H_i}$ which equals $0$ on $H_i$ and $1$ off $H_i$. By Proposition \ref{prop:rho_properties}, the function 
\[
\rho_{\A}+k = \rho_{H_1} + \ldots + \rho_{H_n} + k
\]
is also a directional degree function for every $k \geq 0$. Notice that, for a line $h \in V$,
\[
\rho_{\A}(h) = \textrm{ number of hyperplanes in $\A$ not containing $h$}.
\]
As remarked in Section \ref{subsec:hyparr}, this is precisely the directional degree function associated to the polynomial $l(\A)=l_1\cdots l_n \in \C[V^*]$.

The corresponding power ideal in $\C[V]$ is
\[
I_{\A,k} := I(\rho_\A + k)= \left<h^{\rho_\A(h)+k+1} \mid h \in V, h \ne 0\right>.
\]
We will study this ideal for $k \geq -2$, and show some difficulties that arise for $k \leq -3$.

One can also define the (a priori smaller) ideal
\[
I'_{\A,k} := \left<h^{\rho_\A(h)+k+1} \mid h \textrm{ is a line of the arrangement }\A \right>,
\]
where $h$ ranges only over the one-dimensional intersections of the hyperplanes in $\A$. 
In the special cases $k=-2, -1,0$, these ideals have received considerable attention \cite{Ar03, DM, DP, HR, PS, PSS, SX, Wa}. As mentioned in the introduction, they arise in problems of multivariate polynomial interpolation, in the study of fat point ideals, and in the study of zonotopal Cox rings, among others. In Theorem \ref{th:I=I'} we will prove that $I_{\A,k}=I'_{\A,k}$ in these three important special cases (clearly $I_{\A,k} \supseteq I'_{\A,k}$ in general). We will also show that $I_{\A,k}$ is in some sense better behaved than $I'_{\A, k}$. We will therefore focus our attention on the ideals $I_{\A, k}$.

The inverse system of $I_{\A,k}$ is the $\C[V]$-submodule
\[
C_{\A,k} := C(\rho_\A+k) = \left\{f(x) \in \C[V] \mid h\left(\dx\right)^{\rho_\A(h)+k+1} f(x) = 0 \textrm{ for all } h \in V, h \ne 0\right\}
\]
of $\C[V^*]$, graded by degree; $\C[V]$ acts on it by differentiation. It consists of the polynomials $f$ whose degree along a line is less than or equal to $k$ plus the degree of $l(\A)$ along that line.


\begin{figure}[h]
\begin{center}
\includegraphics[width=2.5in]{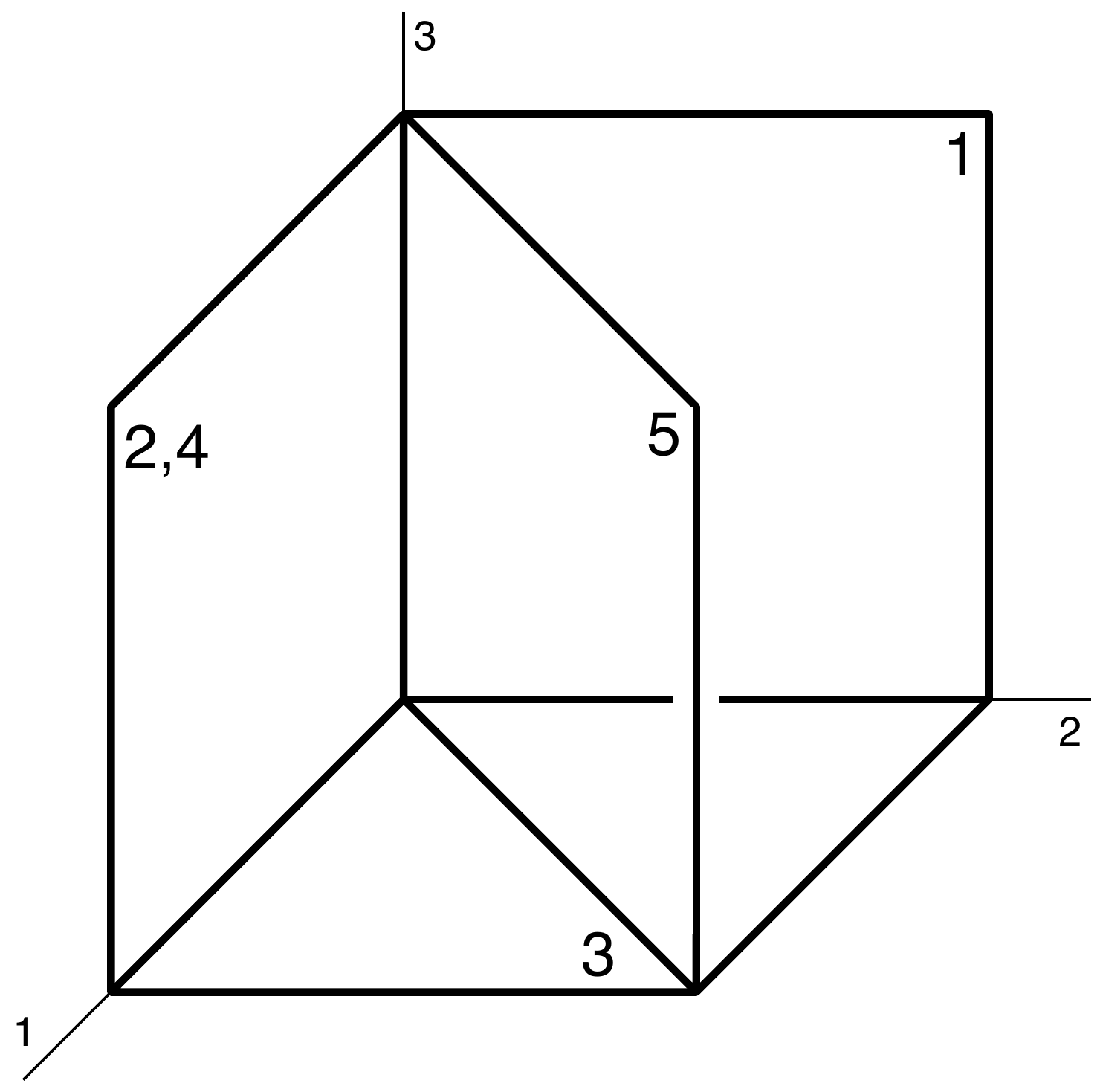} 
\caption{An arrangement of hyperplanes in three dimensions.}
\label{fig:arr1}
\end{center}
\end{figure}

\begin{example}\label{ex}
Let $\G$ be the three-dimensional arrangement of hyperplanes of Figure \ref{fig:arr1}, determined by the linear forms $l_1 = y_1, l_2 = y_2, l_3=y_3, l_4=y_2,  l_5=y_1-y_2$, and let $k=0$. Then

\[
I'_{\G,0} = \left<x_1^3, x_2^4, x_3^2, (x_1+x_2)^4 \right>
\]

and 

\begin{eqnarray*}
I_{\G,0} &=& \left<x_1^3, x_2^4, x_3^2, (x_1+x_2)^4, \right. \\
& & \left. (x_1+ax_2)^5, (x_1+bx_3)^4, (x_2+cx_3)^5, (x_1+dx_2+ex_3)^6 \right>
\end{eqnarray*}
where $a,b,c,d,e$ range over the complex numbers. Simplifying each generator on this list by the previous ones, 
\[
I_{\G,0} = \left<x_1^3, x_2^4, x_3^2, 6x_1^2x_2^2+4x_1x_2^3, 0, 0, 0, 0\right>.
\]
For example, the only monomial in $(x_1+ax_2)^5$ which is not generated by $x_1^3$ and $x_2^4$ is $x_1^2x_2^3$, which is generated as $x_2(6x_1^2x_2^2+4x_1x_2^3) - 4x_1(x_2^4)$. (In particular this means that $I_{\G,0} = I'_{\G,0}$.) Thus
\begin{eqnarray*}
C_{\G, 0} &=& \span(1; \,\,\,\, y_1,y_2,y_3; \,\, \,\, y_1^2, y_2^2, y_1y_2,y_1y_3, y_2y_3;  \\
& & y_2^3, y_1^2y_2, y_1^2y_3, y_1y_2^2,y_2^2y_3, y_1y_2y_3; \,\,\,\, y_1y_2^3-y_1^2y_2^2,  y_2^3y_3, y_1^2y_2y_3, y_1y_2^2y_3; \\
& & y_1y_2^3y_3-y_1^2y_2^2y_3.)
\end{eqnarray*}
and
 \[
\Hilb(C_{\G, 0};q) = 1+3q+5q^2+6q^3+4q^4+q^5. \hfill
\]
\end{example}
%

Our next example shows that the ideals $I_{\A, k}$ and $I'_{\A, k}$ are generally different for $k \geq 1$.

\begin{example}
Let $\H$ be the arrangement in $\R^2$ determined by the linear forms $y_1$ and $y_2$ in $\R^2$. Then
\[
I_{\H,k} = \left<x_1^{k+2}, x_2^{k+2}, (x_1+ax_2)^{k+3} \, | \, a \in \R \right>, \qquad I'_{\H,k} = \left<x_1^{k+2}, x_2^{k+2} \right>.
\]
If we choose $k+4$ different values of $a$, the resulting polynomials $(x_1+ax_2)^{k+3}$ in $I_{\H, k}$ will linearly span all polynomials of degree $k+3$. In $I'_{\H,k}$, on the other hand, the degree $k+3$ component is spanned by $x_1^{k+3}, x_1^{k+2}x_2, x_1x_2^{k+2}, x_2^{k+3}$. These only coincide for $k=-2, -1, 0$.
\end{example}

\begin{proposition} \label{prop:coloops} Let $\H_n^{m+n}$ be an arrangement of $n$ generic hyperplanes in $\C^{m+n}$ and let $k \geq 0$.  Then 
\[
\Hilb(C_{\H_n^m,k};q) = \left[z^k\right] \left(1+\frac{q}{1-qz}\right)^n \frac{(1-qz)^{-m}}{(1-z)}
\] 
\end{proposition}

\begin{proof}
Let $N=m+n$. Fixing a basis $x_1, \ldots, x_n, X_1, \ldots, X_m$ for $V = \C^{N}$, we can assume that the hyperplanes are $x_1=0, \ldots, x_n=0$.
Then
\begin{eqnarray*}
I_{\H_n^m, k} &=& \left< (a_1x_{i_1} + \cdots + a_t x_{i_t}+b_1X_1 + \cdots + b_m X_m)^{t+k+1} \, \mid\,  \right.\\
& & \{i_1, \ldots, i_t\} \subseteq [n], a_1 \ldots, a_t \in \C^*,  b_1, \ldots, b_m \in \C. \Big> 
\end{eqnarray*}
By fixing $x_{i_1}, \ldots, x_{i_t}$ and varying $a_1, \ldots, a_t, b_1, \ldots, b_m$, these powers of linear forms generate every monomial $x_{i_1}^{\alpha_1} \cdots x_{i_t}^{\alpha_t}X_1^{\beta_1}\cdots X_m^{\beta_m}$ 
of degree $t+k+1$, and those with some $\alpha_i=0$ are generated by a smaller such monomial. Therefore
\[
I_{\H_n^m, k} = \left< x_{i_1}^{\alpha_1} \cdots x_{i_t}^{\alpha_t} \, X_1^{\beta_1}\cdots X_m^{\beta_m} 
\, \mid \,
\alpha_i >0, \, \textstyle \sum \alpha_i+ \sum \beta_j = t+k+1 \right>
\]
and, with respect to the dual basis $y_1, \ldots, y_n, Y_1, \ldots, Y_m$ of $V^*$,
\[
C_{\H_n^m, k} = \span \left( y_{i_1}^{\alpha_1} \cdots y_{i_t}^{\alpha_t} \, Y_1^{\beta_1}\cdots Y_m^{\beta_m} 
 \mid 
 \alpha_i >0, \, \textstyle 
 \sum \alpha_i+ \sum \beta_j \leq t+k\right).
\]

Let us count the monomials in $C_{\H_n^m, k} $ of degree $s+t$ which involve exactly $t$ variables among $y_1, \ldots, y_n$. We have $s\leq k$ necessarily. There are ${n \choose t}$ choices for the variables, and ${s+t+m-1\choose s}$ ways to write $s = \sum(\alpha_i-1) + \sum \beta_j$ as a sum of $t+m$ nonnegative integers. Therefore
\[
\Hilb(C_{\H_n^m,k};q) = \sum_{t=0}^n \sum_{s=0}^k {n \choose t} {s+t+m-1 \choose s} q^{s+t}.
\] 
Since $\sum_{t=0}^k {s+t+m-1 \choose s} q^t$ is the coefficient of $z^k$ in $(1-qz)^{-(m+t)}/(1-z)$, we can rewrite this as
\[
\Hilb(C_{\H_n^m,k};q) = \left[z^k\right] \frac{(1-qz)^{-m}}{1-z} \sum_{t=0}^n {n \choose t} \left(\frac{q}{1-qz} \right)^t,
\] 
which gives the desired result.
\end{proof}

\subsection{Deletion and contraction}

We now recall the operations of deletion and contraction.
Suppose that hyperplane $H_1$ in $\A=\{H_1, \ldots, H_n\}$ is not a loop or coloop. The \emph{deletion} of $H_1$ in $\A$ is the arrangement $\A \backslash H_1=\{H_2, H_3, \ldots, H_n\}$ in $V$. The corresponding linear forms are $l_2, \ldots, l_n$ in $V^*$. The \emph{contraction} of $H_1$ in $\A$ (also called the \emph{restriction} of  $\A$ to $H_1$) is the arrangement $\A/H_1 := \{H_1 \cap H_2, H_1 \cap H_3, \ldots, H_1 \cap H_n\}$ in $H_1$. The corresponding linear forms are the images of $l_2, \ldots, l_n$ in the quotient vector space $V^*/l_1 \simeq H^*$. 

\begin{proposition}\label{prop:exact}
Let $\A$ be a hyperplane arrangement and $k \geq -2$.\footnote{For $k=-2$ we need to assume that $\A, \A \backslash H$ and $\A/H$ have no coloops.}
\begin{enumerate}
\item
If $H \in \A$ is not a loop, then there is an exact sequence
\[
0 \rightarrow C_{\A \backslash H, k}(-1) \rightarrow C_{\A,k} \rightarrow C_{\A / H, k} \rightarrow 0
\]
of graded $\C$-vector spaces. Here $C_{\A \backslash H, k}(-1)$ denotes the vector space $C_{\A \backslash H, k}$ with degree shifted up by one.
\item
If $H \in \A$ is a loop, then $ C_{\A,k} = C_{\A \backslash H, k}$.
\end{enumerate}
\end{proposition}

\begin{proposition}\label{prop:generate}
Let $\A=\{H_1, \ldots, H_n\}$ be a hyperplane arrangement in $V$ with corresponding linear forms $l_1, \ldots, l_n$ in $V^*$, and let $k \geq 0$.
\begin{enumerate}
\item
For $k \geq 0$, the space $C_{\A, k}$ is spanned by the polynomials $fl_S = f \prod_{s \in S} l_s$, where $f$ is a polynomial in $\C[V^*]$ of degree at most $k$ and $S$ is a subset of $[n]$.
\item
For $k =-1$, the space $C_{\A, -1}$ is spanned by the polynomials $l_S = \prod_{s \in S} l_s$, where $S$ is a subset of $[n]$ such that $[n] -  S$ has full rank.
\item
For $k =-2$, the space $C_{\A, -2}$ is spanned by the polynomials $l_S=\prod_{s \in S} l_s$, where $S$ is a subset of $[n]$ such that $[n] - S - x$ has full rank for all $x \notin S$.
\end{enumerate}
\end{proposition}

\begin{proof}[Proof of Propositions \ref{prop:exact} and \ref{prop:generate}] In what follows, we will use the description of $C_{\A,k}$ as the set of polynomials $f$  in $\C[V^*]$ 
whose degree $\rho_f(h)$ on a generic line parallel to $h \in V$ is bounded above by $\rho_\A(h)+k = \rho_{l(\A)}(h)+k$, where $l(\A)$ is the defining polynomial of $\A$. For the polynomials in $C_{\A \backslash H,k}$ the bounds are the same along directions contained in $H$, and they are decreased by one along directions not contained in $H$. For the polynomials in $C_{\A / H,k}$, the bounds are the same, but only concern the directions contained in $H$, where these polynomials are defined.

We will prove Propositions \ref{prop:exact} and \ref{prop:generate} in a joint induction on the number of hyperplanes which are neither loops nor coloops. 
We will first settle the case $k \geq 0$.
The base case is a hyperplane arrangement consisting of only loops and coloops. A loop in a hyperplane arrangement in $V$ is the ``hyperplane" $V$ with linear form $0 \in V^*$; it is not noticed by $I_{\A,k}$ and $C_{\A,k}$ and can be safely ignored. Modulo a change of basis, the hyperplane arrangements with only coloops are the arrangements $\H_n^m$. As seen in Proposition \ref{prop:coloops}, $C_{\H_n^m,k}$ is generated by the monomials of the form $y_{i_1}^{\alpha_1} \cdots y_{i_t}^{\alpha_t}Y_1^{\beta_1}\cdots Y_m^{\beta_m}$ with $\alpha_i \geq 1$ and $\sum \alpha_i + \sum \beta_j \leq t+k$. Such a monomial can be rewritten as $f y_{i_1} \cdots y_{i_t}$ where $f$ 
has degree $\leq k$.

Now  suppose that $\A$ is an arrangement and $H$ is a hyperplane of $\A$ which is not a loop or coloop, and suppose that Propositions \ref{prop:exact} and \ref{prop:generate} are true for $\A \backslash H$ and $\A/H$. There is no loss of generality in assuming that $H$ is the first coordinate hyperplane, and $y_1$ is the corresponding linear functional.

By the previous discussion on $C_{\A \backslash H,k}$ and $C_{\A / H,k}$, we have maps
\[
0 \rightarrow C_{\A \backslash H, k}(-1) \xrightarrow{\cdot y_1} C_{\A,k} \xrightarrow{y_1=0} C_{\A / H, k} \rightarrow 0
\]
given by multiplying by $y_1$, and setting $y_1=0$, respectively. Injectivity on the left is immediate. 

To prove exactness in the middle, notice that a polynomial $f$ in $C_{\A,k}$ which maps to $0$ must be a multiple of $y_1$; say $f=y_1g$. To check that $g \in C_{\A \backslash H, k}$ we verify directional degrees.
Since $\rho_{y_1g}(h) \leq \rho_{l(\A)}(h) +k = \rho_{y_1l(\A \backslash H)}(h) +k$ for any direction $h$, it follows that $\rho_{g}(h) \leq \rho_{l(\A \backslash H)}(h) +k$ for any direction $h$.
%

To prove exactness on the right, we use the inductive hypothesis that $C_{\A /H, k}$ is spanned by the products $fl'_S = f \prod_{s \in S} l'_s$, where $f \in \C[H^*]$ of degree $\leq k$, $S$ is a subset of $\{2, \ldots, n\}$, and $l'_s$ is the image of $l_s$ in $H^*$. But this is the image of $fl_S = f \prod_{s \in S} l_s$, which is in $C_{\A,k}$. This proves Proposition \ref{prop:exact} for $\A$.

To prove Proposition \ref{prop:generate} for $\A$ notice that the products $fl_S$ involving $l=y_1$ are the images of the products which generate $C_{\A \backslash H, k}$, while the products $fl_S$ not involving $l=y_1$ map to the generators of $C_{\A / H, k}$. The desired result then follows from Proposition \ref{prop:exact} for $\A$ and the fact that a short exact sequence of vector spaces splits. 

\medskip

For $k=-1, -2$, the proof works in essentially the same way. One needs to be careful about the initial case of the induction, and to adapt the argument in the previous paragraph, as follows.

The initial case of the induction for $k=-1$ is still the arrangement $\H_n^m$, for which $C_{\H_n^m, -1}=\span(1)$, which agrees with the fact that only the set $[n]$ has full rank. When $k=-2$, $C_{\A, -2}$ is only defined when $\A$ has no coloops, so our initial case is the rank $n$ arrangement of $n+1$ generic hyperplanes, where our claim is easily verified. 

By the inductive step of Proposition \ref{prop:generate} for $k = 0$, the  products $l_S$ involving $l=y_1$ are the images of the generators of $C_{\A \backslash H, 0}$, while the products $l_S$ not involving $l=y_1$ map to the generators of $C_{\A / H, 0}$. One then needs to refine these statements by easily checking that they are compatible with the conditions of $[n]-S$ having full rank (for $k=-1$), and $[n]-S-x$ having full rank for all $x$ (for $k=-2$). 
\end{proof}

\subsection{Hilbert series}

Our next goal is to prove that $\Hilb(C_{\A,k};q)$ is an invariant of the matroid $M(\A)$ and the ``excess" dimension $m=\dim V - r(M(\A))$ between the vector space $V$ that $\A$ lives in and the rank of $\A$. It is important to observe that this quantity does depend on $m$. For instance, the arrangements $\H_n^m$ of Proposition \ref{prop:coloops} all have the same matroid but different  Hilbert series $\Hilb(C_{\H_n^m,k};q)$.

\begin{proposition}\label{prop:delcontr}
Let $\A$ be a hyperplane arrangement and let $H$ be a hyperplane which is neither a loop nor a coloop. Then
\[
\Hilb(C_{\A,k};q) = q\Hilb(C_{\A \backslash H, k};q) + \Hilb(C_{\A / H,k};q)
\]
for $k \geq -2$.
\end{proposition}

\begin{proof}
This follow immediately from Proposition \ref{prop:exact}.
\end{proof}

\begin{proposition}\label{prop:loopcoloop}
Let $\A$ be a hyperplane arrangement.
\begin{enumerate}
\item
If $H$ is a loop in $\A$ then $\Hilb(C_{\A,k};q) = \Hilb(C_{\A \backslash H, k};q)$.
\item
If $H$ is a coloop in $\A$ then:
\begin{itemize}
\item
$\Hilb(C_{\A,0};q) = (1+q)\Hilb(C_{\A/H, 0};q)$.
\item
$\Hilb(C_{\A,-1};q) = \Hilb(C_{\A/H, -1};q)$.
\item
$\Hilb(C_{\A,-2};q) = 0$.
\end{itemize} 
\end{enumerate}
\end{proposition}

\begin{proof}
The first part follows immediately from Proposition \ref{prop:exact}; let us prove the second. 
We can assume that the intersection of the hyperplanes of $\A \backslash H$ is the line containing $x_1$. 
The formula $\Hilb(C_{\A,k};q) = q\Hilb(C_{\A \backslash H, k};q) + \Hilb(C_{\A / H,k};q)$ still holds, but now $\A \backslash H$ has different excess dimension. This is a difficulty for $k \geq 1$, but we are fortunate when $k=0,-1,-2$.
For $k=0$, polynomials in $C_{\A \backslash H, 0}$ must be constant on the line $x_1$, so they cannot involve the variable $y_1$. Therefore $C_{\A \backslash H, 0} = C_{\A / H, 0}$, and the first statement follows. For $k=-1$, $I_{\A \backslash H,-1}$ contains $x_1^0=1$, so $C_{\A \backslash H, -1}=0$, which proves the second statement. For $k=-2$, $I_{\A, -2}$ contains $x_1^0=1$, so $C_{\A, -2}=0$, and the last statement follows. 
\end{proof}

\begin{definition}
The \emph{Tutte polynomial} of a matroid $M$ with ground set $E$ and rank function $r$ is defined by 
\[
T_M(x,y) = \sum_{A \subseteq E} (x-1)^{r(M)-r(A)}(y-1)^{|A|-r(A)}.
\]
\end{definition}

\begin{definition}
A function $f:\textsf{Mat} \rightarrow R$ from the class of finite matroids to a commutative ring $R$ is said to be a \emph{generalized Tutte-Grothendieck invariant} if there exist $a,b,L,C \in R$, with $a$ and $b$ invertible, such that the following properties hold.
\begin{enumerate}
\item
If $M$ and $N$ are isomorphic matroids then $f(M) = f(N)$.
\item
If $e$ is neither a loop nor a coloop of $M$, then $f(M) = af(M\backslash e) + bf(M/e).$
\item
If $e$ is a loop of $M$ then $f(M) = Lf(M\backslash e)$.
\item
If $e$ is a coloop of $M$ then $f(M) = Cf(M/e)$.
\item
$f(\emptyset) = 1$.
\end{enumerate}
A function $f:\textsf{Mat} \rightarrow R$ is said to be a \emph{weak generalized Tutte-Grothendieck invariant} if it satisfies conditions 1 and 2 above.
\end{definition}

The Tutte polynomial is the universal Tutte-Grothendieck invariant in the following sense.
\begin{proposition}{\rm\cite{Wh}}
Any generalized Tutte-Grothendieck invariant is an evaluation of the Tutte polynomial. With the notation above, $f(M)$ is given by
\[
f(M) = a^{|M|-r(M)}b^{r(M)} T_M\left(\frac{C}{b}, \frac{L}{a}\right).
\]
\end{proposition}

If the Tutte polynomial of $M$ is 
\[
T_M(x,y) = \sum_{i, j \geq 0} b_{ij} x^i y^j,
\]
we define its \emph{umbral Tutte polynomial} to be
\[
\T_M(\t) = \sum_{i, j \geq 0} b_{ij} t_{ij}
\]
where $\t = (t_{ij})_{i,j \geq 0}$ are indeterminates.\footnote{The adjective \emph{umbral} refers to the \emph{umbral calculus}, developed in the 19th century and later made precise and rigorous by Rota. \cite{RR} This method derives identities about certain sequences (such as the sequence of Bernoulli polynomials) by treating the subindices as if they were exponents; it motivates the following results.}
The following proposition is essentially known; for instance, a slightly less general version can be found in \cite[Prop. 6.2.8]{Wh}. For completeness, we include a proof.

\begin{proposition}\label{prop:weakTG}
Any weak generalized Tutte-Grothendieck invariant is an evaluation of the umbral Tutte polynomial. With the notation above, 
\[
f(M) = a^{|M|-r(M)}b^{r(M)} \T_M\left(\frac{f(M_{ij})}{a^jb^i}\right)_{i,j \geq 0}
\]
where $M_{ij}$ is the matroid consisting of $i$ coloops and $j$ loops.
\end{proposition}

\begin{proof}
One way of computing a generalized Tutte-Grothendieck invariant of a matroid $M$ is by recursively building a \emph{computation tree}. The matroid $M$ is at the root of the tree. We choose an element $e$; if it is neither a loop nor a coloop, then we make $M \backslash e$ and $M/e$ the left and right children of $M$, and label the edges $x$ and $y$ respectively. If $e$ is a loop or a coloop, then we make $M \backslash e$ or $M/e$ its only child and we label the edge $L$ or $C$ respectively. We continue this process recursively until every leaf is the empty matroid. Then we add the weights of the leaves, where the weight of a leaf is the product of the labels of the edges between it and the root.

Build a partial computation tree for $f(M)$ by never choosing an element $e$ which is a loop or coloop, and stopping when every leaf is labelled by a matroid of the form $M_{ij}$. This same tree will tell us how to express the Tutte polynomial $T_M(x,y)$occurrence as a linear combination of the Tutte polynomials of the $M_{ij}$s at the leaves. Since $T_{M_{ij}}(x,y) = x^iy^j$, exactly $b_{ij}$ of the leaves of the tree are labelled by $M_{ij}$. To compute $f(M)$, then, it suffices to replace each occurrence of $x^iy^j$ in the Tutte polynomial by $f(M_{ij})$.
\end{proof}

\begin{theorem}\label{th:Tutte}
If $\A$ is a rank $r$ arrangement of $n$ hyperplanes in $V=\C^{r+m}$ and $k \geq 0$ then 
\[
\sum_{k \geq 0} \Hilb(C_{\A,k};q) z^k= \, \frac{q^{n-r}}{(1-z)(1-qz)^m} \,\, T_{\A}\left(1+\frac{q}{1-qz},\frac1q\right).
\]
\end{theorem}

\begin{proof}
Proposition \ref{prop:delcontr} shows that, if we restrict our attention to arrangements of excess dimension $m$, then $\Hilb(C_{\A,k};q)$ is a weak generalized Tutte- Grothendieck invariant on the matroid $M(\A)$. Therefore we can use Proposition \ref{prop:weakTG}, plugging in the formula for $\Hilb(C_{\H_n^m,k};q)$ obtained in Proposition \ref{prop:coloops}. We obtain
\begin{eqnarray*}
\Hilb(C_{\A,k};q) &=& q^{n-r} \T_{\A}\left(\left[z^k\right] \left(1+\frac{q}{1-qz}\right)^i \frac{(1-qz)^{-m}}{(1-z)}\right)_{i,j \geq 0} \\
 &=& q^{n-r} \sum_{i,j \geq 0} b_{ij}\left(\left[z^k\right] \left(1+\frac{q}{1-qz}\right)^i \frac{(1-qz)^{-m}}{(1-z)}\right) q^{-j} 
\end{eqnarray*}
which is equivalent to the given formula. \end{proof}

Taking the limit as $k \rightarrow \infty$, 
\begin{eqnarray*}
\lim_{k \rightarrow \infty} \Hilb(C_{\A,k};q) &=& \lim_{k \rightarrow \infty} [z^0+\cdots+z^k] \,\,\, \frac{q^{n-r}}{(1-qz)^m} \,\, T_{\A}\left(1+\frac{q}{1-qz},\frac1q\right) \\
&=& 
\frac{q^{n-r}}{(1-q)^m} \,\, T_{\A}\left(1+\frac{q}{1-q},\frac1q\right) =
\frac1{(1-q)^{r+m}}
\end{eqnarray*}
since it is easily shown that $T_{\A}\left(\frac{1}{1-q},\frac1q\right) = 1/\left[(1-q)^r q^{n-r}\right]$. This confirms the fact that every polynomial in $\C[V^*]$ is in $C_{\A, k}$ for a large enough value of $k$. 

%


\begin{corollary} \label{cor:hilbertC0} 
If $\A$ is an arrangement of $n$ hyperplanes of rank $r$, then
\[
\Hilb(C_{\A,0};q) = q^{n-r}T_{\A}\left(1+q,\frac1q\right).\]
\end{corollary}

\begin{proof}
Substitute $z=0$ into Theorem \ref{th:Tutte}.
\end{proof}

\begin{proposition} \label{prop:hilbertC1} 
If $\A$ is an arrangement of $n$ hyperplanes of rank $r$, then\[
\Hilb(C_{\A,-1};q) = q^{n-r}T_{\A}\left(1,\frac1q\right).
\]
\end{proposition}

\begin{proof}
This is an easy consequence of Propositions \ref{prop:delcontr} and \ref{prop:loopcoloop}.
\end{proof}

\begin{proposition} \label{prop:hilbertC2} 
If $\A$ is an arrangement of $n$ hyperplanes of rank $r$, then
\[
\Hilb(C_{\A,-2};q) = q^{n-r}T_{\A}\left(0,\frac1q\right).
\]
\end{proposition}
\begin{proof}
This is an easy consequence of Propositions \ref{prop:delcontr} and \ref{prop:loopcoloop}.
\end{proof}

\begin{example}
The arrangement of Example \ref{ex} has Tutte polynomial
\[
T_{\G}(x,y) = x^3+x^2y+x^2+xy^2+xy
\]
which shows that 
\begin{eqnarray*}
\Hilb(C_{\A,0};q) &=& q^2 \left[(1+q)^3 + (1+q)^2/q + (1+q)^2 + (1+q)/q^2 + (1+q)/q \right] \\
&=& 1+3q+5q^2+6q^3+4q^4+q^5
\end{eqnarray*}
confirming our earlier computation.
\end{example}

\begin{theorem}\label{th:I=I'}
If $\A$ is an arrangement and $k \in \{0, -1, -2\}$ then $I_{\A, k} = I'_{\A, k}$.
\end{theorem}

\begin{proof}
This follows since $I_{\A, k}$ contains $I'_{\A, k}$, and the previous propositions show that the Hilbert series of $I_{\A, k}$ is equal to the known Hilbert series of $I'_{\A, k}$ for $k \in \{0, -1, -2\}$. \cite{Ar03, DM, DP, HR, PSS, Wa}
\end{proof}

\subsection{Holtz and Ron's conjectures}

We have now proved Holtz and Ron's conjecture on the internal zonotopal algebra. 

\begin{theorem}
{\rm \cite[Conjecture 6.1]{HR}}\label{conj1}
The inverse system of the ideal $I'_{\A, -2}$ is spanned by the polynomials $l_S=\prod_{s \in S} l_s$, where $S$ is a subset of $[n]$ such that $[n] - S - x$ has full rank for all $x \notin S$.
\end{theorem}

\begin{proof}
Now that we know that $I'_{\A,-2} = I_{\A, -2}$, this is exactly Proposition \ref{prop:generate}.3.
\end{proof}

A set $X$ of integer vectors in $\R^d$ is \emph{unimodular} if its $\Z$-span contains all the integer vectors in its $\R$-span. Define the \emph{zonotope} $Z(X)$ to be the Minkowski sum of the vectors in $X$, and define the \emph{box spline} $M_X$ to be the convolution product of the uniform measures on the vectors in $X$; this is a continuous piecewise polynomial function on $Z(X)$. Let $\A(X)$ be the arrangement of hyperplanes orthogonal to the vectors in $X$, or equivalently, the arrangement dual to $Z(X)$.

Motivated by the study of box splines, Holtz and Ron \cite{HR} proved Proposition \ref{prop:generate}.1 in the case $k=0$ and Proposition \ref{prop:generate}.2, and conjectured Theorem \ref{conj1}. (Their results really concerned the ideals $I'_{\A,k}$ for $k=0,-1,-2$, but now we know that $I'_{\A,k} = I_{\A,k}$ in these cases.)
As they remarked, Theorem \ref{conj1} also implies their conjecture on the spline interpolation of functions on the lattice points inside a zonotope:

\begin{corollary} {\rm \cite[Conjecture 1.8]{HR}}
Let $X$ be a unimodular set of vectors, let $Z_{-}(X)$ be the set of integer points inside the zonotope $Z(X)$ and let $M_X$ be the box spline of $X$. Any function on $Z_{-}(X)$ can be extended to a function on $Z(X)$ of the form $p(\frac{\partial}{\partial \x}) M_X$ for a unique polynomial $p \in C_{\A(X),-2}$.
\end{corollary}
%

\subsection{A basis for $C_{\A, k}$.}

Fix a linear order on the hyperplanes of $\A$. For a basis $B$, say an element $i \in B$ is \emph{internally active in $B$} if $B$ is the lexicographically smallest basis containing $B-i$. Similarly, say an element $e \notin B$ is \emph{externally active in $B$} if $B$ is the lexicographically largest basis in $B\cup e$. Let $I(B)$ and $E(B)$ denote the sets of internally and externally active elements in $B$, respectively. Say a basis $B$ is \emph{internal} if $I(B) = \emptyset$. We will need the following facts:

\begin{proposition}\label{prop:activities} {\rm\cite{Wh}} Let $M$ be a matroid with a linear order on its elements.
\begin{enumerate}
\item As $B$ ranges over all bases of $M$, the intervals $[B-I(B), B \cup E(B)]$ partition the set $2^M$; in other words, every subset of a matroid can be written uniquely as $B-I \cup E$ for some basis $B$ and some $I \subseteq I(B)$, $E \subseteq E(B)$. 
\item
If $B$ is a basis, $I \subseteq I(B)$ and $E \subseteq E(B)$, then $r(B-I \cup E) = r-|I|$.
\item The Tutte polynomial of $M$ is 
\[
T_M(x,y) = \sum_B x^{|I(B)|} y^{|E(B)|}
\]
summing over all bases of $M$.
\end{enumerate}
\end{proposition}

\begin{proposition}\label{prop:basis}
Let $\A$ be an arrangement.
\begin{enumerate}
\item
For $k \geq 0$, a basis for $C_{\A,k}$ is given by the set $L_k$ of $l$-monomials of the form
\[
m_{B, I, \alpha_I} = \prod_{i \in \A-B-E(B)} l_i\,\,  \prod_{j \in I} l_j^{\alpha_j+1} 
\]
where $B$ is a basis of $\A$, $I \subseteq I(B)$ is a subset of the internally active elements of $B$, and $\alpha_I = (\alpha_i)_{i \in I}$ is a sequence of non-negative integers with $\sum \alpha_i \leq k$.
\item
For $k=-1$, a basis for $C_{\A,-1}$ is given by the $l$-monomials $l_{\A-B-E(B)}$ 
where $B$ is a basis of $\A$.
\item
For $k=-2$, a basis for $C_{\A,-2}$ is given by the $l$-monomials $l_{\A-B-E(B)}$  
where $B$ is an internal basis of $\A$.
\end{enumerate}
\end{proposition}

\begin{proof}
We start with the case $k \geq 0$. First we prove that $L_k$ spans $C_{\A,k}$. 

From Proposition \ref{prop:generate} and the fact that the $l_i$s span $V^*$, it follows that $C_{\A,k}$ is spanned by the $l$-monomials inside it. Define the total order $<$ on the supports of the $l$-monomials by declaring $\supp(a)>\supp(b)$ if 
\begin{enumerate}
\item
$|\supp(a)| > |\supp(b)|$, or
\item
$|\supp(a)| = |\supp(b)|$ and $\supp(a)>\supp(b)$ in reverse lexicographic order.
\end{enumerate}
Among the $l$-monomials in $C_{\A,k}$ which are not in the span of $L_k$,
let $m$ be one having maximal support $S$ according to the order $<$.
By Proposition \ref{prop:activities}.1, this support can be written uniquely as
\[
S=(\A - B - E) \,\cup \,I,  \qquad \textrm{ for some basis } B \textrm{ and some } E \subseteq E(B), I \subseteq I(B). 
\]
In fact, we claim that $E=E(B)$.

Suppose $E \neq E(B)$ and take $e \in E(B) - E$. Then $e$ is the smallest element in the unique circuit $C \cup e$ contained in $B\cup e$. Write $l_e = \sum_{c \in C} a_cl_c$. All elements of $C$ are larger than $e$ and hence not  internally active in $B$. We have
\[
m = \sum_{c \in C} a_c ml_c/l_e.
\]
Notice that $l_e$ is one of the factors of $m$ by definition. Each $ml_c/l_e$ has degree $1$ in $l_c$, so it is in $C_{\A, k}$, and has support larger than $S$. So each term in the right hand side is spanned by $L_k$, which contradicts the assumption that $m$ is not in the span of $L_k$. It follows that the set $S$ is indeed of the form
\[
S=(\A-B-E(B)) \cup I.
\] 

Now let $N_S$ consist of those monomials with support $S$ which are in $C_{\A,k}$ and not generated by $L_k$.
Consider $m' \in N_S$ having lowest total degree in the variables indexed by $(\A-B-E(B))$. At least one of those variables must be raised to a power greater than 1; say it is $l_e$. Since $B$ is a basis of $\A$ we can write $l_e = \sum_{b \in B} a_bl_b$ and obtain 
\[
m' = \sum_{b \in B} a_bm'l_b/l_e. 
\]
If $b \in I$ then $a_bm'l_b/l_e$ has the same support $S$ and lower $(\A-B-E(B))$-degree than $m'$, so it is spanned by $L_k$. If $b \notin I$  then $a_bm'l_b/l_e$ has support larger than $S$ so it is spanned by $L_k$. So each term in the right hand side is spanned by $L_k$, which contradicts the assumption that $m'$ is not in the span of $L_k$. We conclude that $N_S$ is empty, and $L_k$ spans $C_{\A,k}$. 

Finally we claim that the number of monomials in $L_k$ equals the dimension of $C_{\A,k}$. Consider a basis $B$ with internal activity $|I(B)|=i$ and external activity $|E(B)|=e$. For some $a \leq i$ and $b \leq k$, we need to choose $a$ internally active elements and $a$ non-negative $\alpha$s adding up to $b$; there are ${i \choose a}{a+b-1 \choose a-1}$ choices. The resulting monomial has degree $|\A|-|B|-|E(B)|+a+b = n-r-j+a+b$. Comparing this with the second equation in the proof of Theorem \ref{th:Tutte}, we can check that we have found the correct number of generators in each degree. It follows that $L_k$ is a basis for $C_{\A,k}$.

\medskip

Next we proceed with the case $k=-1$. The space $C_{\A, -1}$ is spanned by the monomials $l_S$ with $r(\A-S)=r$; by Proposition \ref{prop:activities}, this is equivalent to $S=\A-B-E$ for some basis $B$ and $E \subseteq E(B)$. Repeating the argument above, we find that the monomials $l_S$ with $S=\A-B-E(B)$ are sufficient to span $C_{\A, -1}$. Since the dimension of $C_{\A, -1}$ equals the number of bases of $\A$, these monomials are a basis for $C_{\A,-1}$

\medskip

Finally we settle the case $k=-2$. A monomial $l_S$ with $S=\A-B-E$ is in $C_{\A, -2}$ if and only if $B$ is an internal basis, because the number $r(\A-S-x)=r(B \cup E - x)$ equals $r-1$ if $x \in I(B)$, and $r$ otherwise. As above the monomials $l_{\A-B-E(B)}$, with $B$ an internal basis, form a spanning set for $C_{\A, -2}$; this is actually a basis since the dimension of this space is the number of internal bases of $\A$.
\end{proof}

Holtz and Ron also mention the ``inherent difficulties we encountered in the internal study due to the absence of a `canonical' basis for" $C_{\A, -2}$. The case $k=-2$ of Proposition \ref{prop:basis} offers a satisfactory solution to these difficulties.


\subsection{The space $C_{\A, k}$ for $k \leq -3$}

We do not know whether the ideals $I_{\A, k}$ are well-behaved for $k \leq -3$ in general. To compute $\Hilb(C_{\A,  k})$ for $k \geq -2$, we recursively 
\begin{itemize}
\item
produced an upper bound for the Hilbert series by deletion-contraction, and 
\item
constructed a large set of polynomials inside $C_{\A, k}$, all of which were monomials in the $l_i$s. 
\end{itemize}
In the cases $k=0,-1,-2$, the existing proofs \cite{DM, DP, HR, Wa}
are different from ours, but they all rely on constructing a large set of polynomials inside $C_{\A, k}$ which is spanned by monomials in the $l_i$s.

These approaches will not work for $k \leq -3$, because in that case $C_{\A, k}$ is not necessarily spanned by $l_i$-monomials, as the following example shows.

\begin{example}
Consider the arrangement $\G$ of hyperplanes in $\C^3$ given by the linear forms $l_1=y_1+y_2, l_2 = y_2, l_3=-y_1+y_2, l_4 =y_1+y_3, l_5=y_3, l_6=-y_1+y_3$; a real picture of this arrangement is shown in Figure \ref{fig:arr2}, when intersected with the hemisphere $y_1 >0$ of the unit sphere. Then 
\[
I_{\G,-3} = \left<x_1^2, x_2^1, x_3^1\right>
\]
and 
\[
C_{\G, -3} = \span(1, y_1).
\]
Notice that the space $C_{\G, -3}$ is \textbf{not} spanned by monomials in the $l_i$s. Similar examples exist for $k < -3$. 
\end{example}

\begin{figure}[h]
\begin{center}
\includegraphics[width=2.5in]{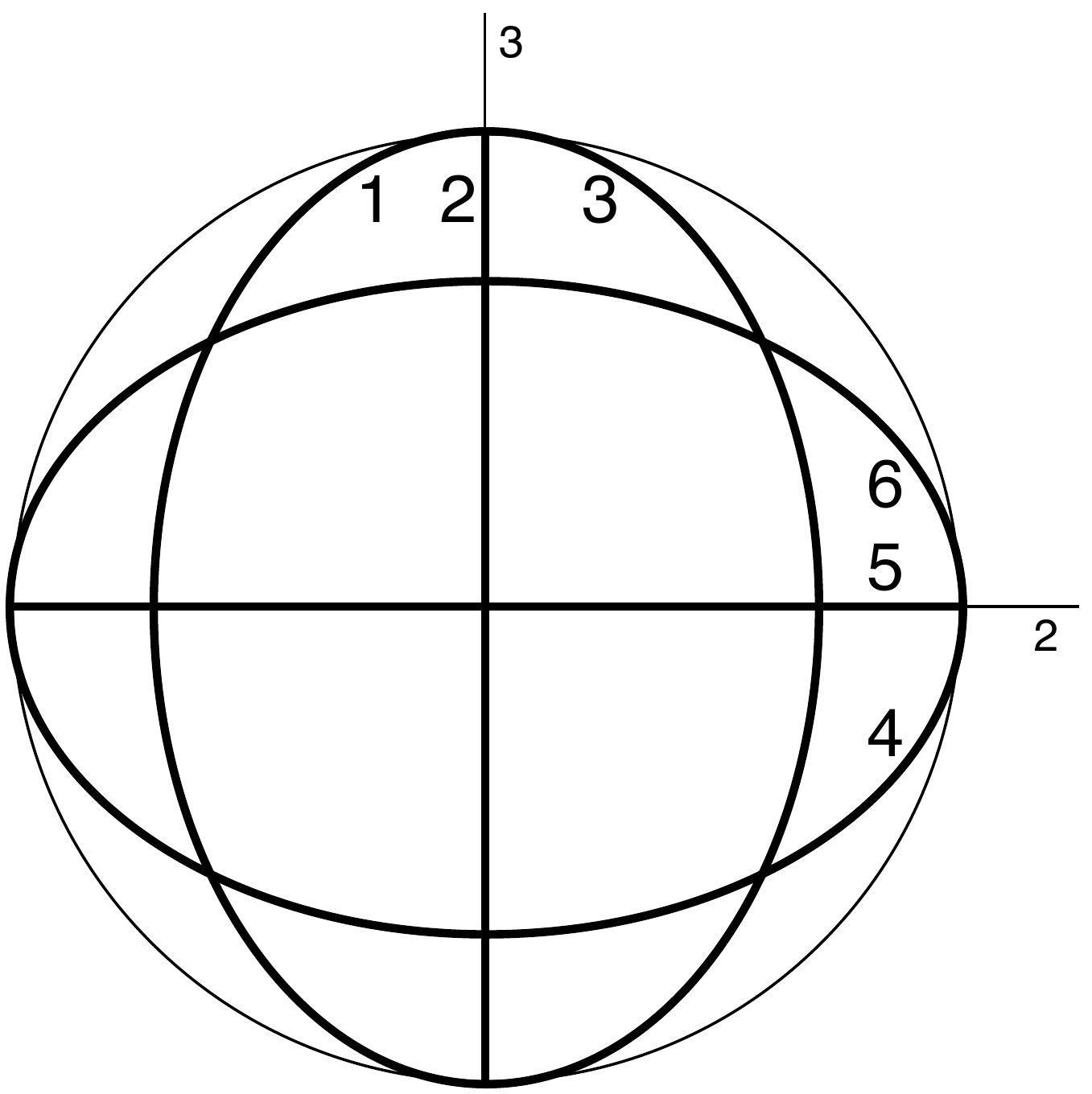} 
\caption{An arrangement $\G$ such that $C_{\G, -3}$ is not spanned by $l_i$-monomials.}
\label{fig:arr2}
\end{center}
\end{figure}

\section{Fat point ideals}\label{sec:fatpoints}

The results of Section \ref{sec:hyparr} are closely related to the theory of fat point ideals. We now outline the connection and apply our results to that theory.

Given a function $\sigma: \P V \rightarrow \N$, we define the \emph{fat point ideal} of $\sigma$ to be the ideal of polynomials in $\C[V^*]$ which vanish at $p$ up to order $\sigma(p)$. We denote it
\[
J({\sigma}) = \bigcap_{p \in \P V} m_p^{\sigma(p)}
\]
where $m_p$ is the ideal of polynomials vanishing at $p$. When $\sigma$ takes negative values, we will simply define $J(\sigma) = J(\widehat{\sigma})$ where $\widehat{\sigma}(v) = \max(\sigma(v), 0).$

The \emph{inverse system} of $J(\sigma)$ is the submodule 
\[
J(\sigma)^{-1} = \left\{f(x) \in \C[V] \mid g\left(\dx\right) f(x) = 0 \textrm{ for all } g \in J_{\sigma}\right\}
\]
of $\C[V]$, which $\C[V^*]$ acts on by differentiation.

Fat point ideals have been studied extensively in the finite case: Given finitely many points $P=\{p_1, p_2, \ldots, p_s\} \subset \P V$ and positive integers $N = (n_1, \ldots, n_s)$, the \emph{fat point ideal of $P$ and $N$} is the homogeneous ideal of polynomials vanishing at $P_i$ up to order $n_i$. We refer the reader to \cite{EI, Ha, Sc} for more information.

Fat point ideals are closely connected to power ideals due to the following theorem of Emsalem and Iarrobino:

%
%

\begin{theorem}\label{th:iarrobino}{\rm\cite{EI}} For any $\sigma: \P V \rightarrow \N$,
\[
(J({\sigma})^{-1})_i = (I_{i-\sigma})_i
\]
\end{theorem}

\begin{corollary} \label{cor:JC} For any $\sigma: \P V \rightarrow \N$
\[
(J({\sigma}))_i = C(i-\sigma)_i
\]
\end{corollary}

\begin{proof} 
This is simply a restatement of Lemma \ref{lemma:degreeorder}.
\end{proof}

Our results on power ideals of hyperplane arrangements allow us to compute the Hilbert series of a family of fat point ideals associated to a hyperplane arrangement. Let $\A=\{H_1, \ldots, H_n\}$ is an arrangement of $n$ hyperplanes in a vector space $V$ and let $l_1, \ldots, l_n \in V^*$ be the corresponding linear forms. For each $p \in V$ let $f_{\A}(p)$ be the number of hyperplanes of $\A$ containing $p$. This can be regarded as a function $f_{\A}:  \P V \rightarrow \N$. Notice that $f_{\A}(p)$ is the order of vanishing of the polynomial $l_1\cdots l_n$ at $p$.

We will consider the shifts $f_{\A}-k$ of this function by a constant $k$. 
This is mostly interesting for $0 \leq k \leq n$: for $k <0$ our function is positive everywhere so $J(f_{\A}-k)=\{0\}$, while for $k > n$ our function is negative everywhere so $J(f_{\A}-k) = \C[V^*]$.
We are interested in the filtration 
\[
\{0\}= J(f_{\A}+1) \subseteq J(f_{\A}) \subseteq J(f_{\A}-1) \subseteq \cdots \subseteq J(f_{\A}-n) = \C[V^*],
\]
of the space of polynomials in $\C[V^*]$ by how the order of vanishing of a polynomial compares to the order of vanishing of $l_1 \cdots l_n$.

\begin{theorem}\label{th:fatpoints}
Let $\A$ be a rank $r$ arrangement of $n$ hyperplanes in $V = \C^{r+m}$. For each $p \in V$ let $f_{\A}(p)$ be the number of hyperplanes of $\A$ containing $p$. Consider the filtration of $\C[V^*]$:
\[
\{0\}= J(f_{\A}+1) \subseteq J(f_{\A}) \subseteq J(f_{\A}-1) \subseteq \cdots \subseteq J(f_{\A}-n) = \C[V^*],
\]
where
\[
J(f_\A-k) := \left\{ \textrm{polynomials in } \C[V^*] \textrm{ vanishing at $p$ to order }f_{\A}(p) -k\right\}.
\]
If $J_{\A, k} = J(f_{\A} - k) / J(f_{\A}-(k-1))$, then
\[
\sum_{k=0}^n \Hilb(J_{\A, k}; t) \, s^k = \frac{s^n}{(1-t)^m} T_\A\left(\frac{2-t}{1-t}, \frac{2s-t}s \right).
\]
\end{theorem}

\begin{proof}
First of all notice that the polynomials in $J(f_{\A}-k)$ must vanish up to order $n-k$ at the origin, so this ideal cannot contain polynomials of degree less than $n-k$. For larger degrees, \emph{i.e.}, for $i \geq -k$, using Corollary \ref{cor:JC} we have that $(J(f_{\A}-k))_{n+i} = C(n+i+k-f_{\A})_{n+i} = C(\rho_{\A}+i+k)_{n+i} = (C_{\A, i+k})_{n+i}$. Since $i+k \geq 0$, Theorem \ref{th:Tutte} then gives
\begin{eqnarray*}
\Hilb(J(f_{\A}-k);t) &=& \sum_{i \geq 0} \dim (C_{{\A},i+k})_{n+i} \, t^{n+i}\\
&=& t^n \sum_{i \geq 0} \, t^i \, [q^iz^i] \, \frac{q^{-r}z^{-k}}{(1-z)(1-qz)^m} \, T_{\A}\left(1+\frac{q}{1-qz}, \frac1q \right)
\end{eqnarray*}
Notice that if $U(q,z) = \sum a_{ij}q^iz^j$ is a formal power series in two variables, then $[q^0]U(q,t/q) = \sum a_{ii} t^i$. Therefore
\begin{eqnarray*}
\Hilb(J(f_{\A}-k);t) &=& t^n\, [q^0]\, \frac{q^{-r}(t/q)^{-k}}{(1-t/q)(1-t)^m}\, T_{\A}\left(1+\frac{q}{1-t}, \frac1q \right)\\
&=& \frac{t^{n-k}}{(1-t)^m} \, [q^{n-k}]\, \frac1{1-\frac{t}{q}}\,T_\A\left(\frac{2-t}{1-t}, 2-q\right).
\end{eqnarray*}
In the last step we use the identity $T_\A(1+ax, 1+\frac{y}{a}) = a^{r-n}T_\A(1+x, 1+y)$, which follows easily from the definition of the Tutte polynomial. We get that 
\[
\Hilb(J_{\A,k};t)  = \frac{t^{n-k}}{(1-t)^m} \, [q^{n-k}]\, T_\A\left(\frac{2-t}{1-t}, 2-q\right),
\]
from which the result readily follows. \end{proof}

Notice that $J_{\A,0} = J(f_{\A})$ is the principal ideal generated by the product of $n$ linear forms determining the hyperplanes of $\A$; therefore
\[
\Hilb (J(f_{\A}); t) = \frac{t^n}{(1-t)^{r+m}}
\]
which, one can check, is consistent with the formula of Theorem \ref{th:fatpoints}.

\section{Zonotopal Cox rings}\label{sec:zono}

\subsection{Cox rings and their zonotopal version}

In this section we describe the close relationship between our work and the zonotopal Cox rings defined by Sturmfels and Xu \cite{SX}. Fix $m$ linear forms $h_1, \ldots, h_m$ on an $n$-dimensional vector space $V$, and consider the following family of ideals of $\C[V^*]$:
\[
I_{\u} = \left< h_1^{u_1+1}, \ldots, h_m^{u_m+1} \right>, \qquad \u =(u_1, \ldots, u_m) \in \N^m
\]
Also consider the corresponding inverse systems $I_{\u}^{-1} = \bigoplus_{d \geq 0} I_{d, \u}^{-1}$ in $\C[V]$, graded by degree $d$.

These inverse systems are intimately related to the Cox ring of the variety which one obtains from $\mathbb{P}^{d-1}$ by blowing up the points $h_1, \ldots, h_m$; the relationship is as follows. Let $G$ be the space of linear relations among the $h_i$s. As an additive group, $G$ acts on $R=\C[s_1, \ldots, s_m, t_1, \ldots, t_m]$, with the action of $\lambda \in G$ given by $s_i \mapsto s_i$ and $t_i \mapsto t_i+\lambda_is_i.$
The \emph{Cox-Nagata ring} is the invariant ring $R^G$ with multigrading given by $\deg s_i = e_i$ and $\deg t_i = e_0 + e_i$, where $(e_0, \ldots, e_n)$ is the standard basis for $\Z^{n+1}$.

\begin{theorem}\label{th:iso} {\rm\cite{SX}}
The $\C$-vector spaces $I_{d,\u}^{-1}$ and $R^G_{d, \u}$ are isomorphic.
\end{theorem}

It will be useful for our discussion to describe this isomorphism explicitly; this is done in \cite{SX} and is easily understood in the following example.

\begin{example} Let $h_1=x_1, h_2 = x_2, h_3=x_3, h_4=x_1+x_2$, so $G = \span\{(1,1,0,-1)\}$. 
Consider the ideal $I_{(2,3,1,3)} = \left<x_1^3, x_2^4, x_3^2, (x_1+x_2)^4\right>$, which happens to coincide with the ideal $I_{\G,0}$ of  Example \ref{ex}. Form a matrix $A$ whose columns are the $h_i$s. 
Given a polynomial $f(y_1, y_2, y_3) \in I^{-1}_{4, (2,3,1,3)}$ (which equals $(C_{\G, 0})_4$ in this case), we plug in the vector
\[
A 
\left [\begin{matrix}
t_1/s_1 \\
t_2/s_2 \\
t_3/s_3 \\
t_4/s_4\\
\end{matrix} \right]
=
\left[ \begin{matrix}
1 & 0 & 0 & 1 \\
0 & 1 & 0 & 1 \\
0 & 0 & 1 & 0 \\
\end{matrix} \right]
\left [\begin{matrix}
t_1/s_1 \\
t_2/s_2 \\
t_3/s_3 \\
t_4/s_4\\
\end{matrix} \right]
\]
and multiply by $s_1^{u_1} \cdots s_m^{u_m}$ to obtain the corresponding polynomial in $R^G_{d, \u}$. For example the polynomial $y_1y_2^3-y_1^2y_2^2 \in I^{-1}_{4, (2,3,1,3)}$ maps to
\[
\left[\left(\frac{t_1}{s_1} + \frac{t_4}{s_4}\right) \left(\frac{t_2}{s_2} + \frac{t_4}{s_4}\right)^3 -
\left(\frac{t_1}{s_1} + \frac{t_4}{s_4}\right)^2 \left(\frac{t_2}{s_2} + \frac{t_4}{s_4}\right)^2
\right] s_1^2s_2^3s_3s_4^3.
\]
The expression inside each parenthesis is invariant under the $G$-action because $G$ is orthogonal to the rows of $A$, and the final result is a polynomial since the functions in $I^{-1}_{d, \u}$ vanish at $h_i$ to order at most $u_i$.
\end{example}

Cox rings are the object of great interest, and the computation of their Hilbert series has proved to be a subtle question. Much of the existing literature has focused on the case where the $h_i$ are generic. By contrast, here we are interested in very special configurations of $h_i$s, namely the configurations of lines in a hyperplane arrangement $\A$. We do not expect there to be a simple formula for the Hilbert series in this case. Surprisingly, it is possible to identify a natural subring of the Cox ring, whose Hilbert series we can compute explicitly in terms of the combinatorics of $\A$ only.

\bigskip

Let $\A=\{H_1, \ldots, H_n\}$ be an essential arrangement of $n$ hyperplanes in $V = \C^r$ and let $h_1, \ldots, h_m$ be the lines of this arrangement. Let $H$ be the \emph{non-containment line-hyperplane matrix}; \emph{i.e.}, the $m \times n$ matrix whose $(i,j)$ entry equals $0$ if $h_i$ is on $H_j$, and equals $1$ otherwise. Sturmfels and Xu \cite{SX} define the \emph{zonotopal Cox ring} of $\A$ to be
\[
\ZZ(\A) = \bigoplus_{(d,\a) \in \N^{n+1}} R^G_{(d, H\a)}
\]
and the \emph{zonotopal Cox module of shift $w$} to be
\[
\ZZ(\A, w) = \bigoplus_{(d,\a) \in \N^{n+1}} R^G_{(d, H\a+w)}
\]
for $w \in \Z^n$.\footnote{Sturmfels and Xu denote them by $Z^G$ and $Z^{G,w}$ respectively. Our notation is more accurate because these objects do not depend only on $G$, which determines $R^G$ but does not determine the matrix $H$.} Of particular interest are the \emph{central} and \emph{internal} zonotopal Cox modules  
$\ZZ(\A, -\mathbf{e})$ and $\ZZ(\A, -2\mathbf{e})$, where $\mathbf{e} = (1,\ldots, 1)$. Their names are motivated by the following proposition.

\begin{proposition}\label{prop:zonotopal}
For $\a=(a_1, \ldots, a_n) \in \N^n$ let $\A(\a)$ denote the arrangement consisting of $a_i$ copies of the $i$th hyperplane $H_i$ of $\A$. Then 
\[
R^G_{(d, H\a)} \cong (C_{\A(\a), 0})_d, \,\,\,\,\,\,R^G_{(d, H\a-\mathbf{e})} \cong (C_{\A(\a), -1})_d, \,\,\,\,\,\, R^G_{(d, H\a - 2\mathbf{e})} \cong (C_{\A(\a), -2})_d
\]
as vector spaces.
\end{proposition}

\begin{proof}
Let $k \in \Z$. By Theorem \ref{th:iso} we have that $R^G_{(d, H\a+k\e)} \cong I_{d,H\a+k\e}^{-1}$. 
But the entries of $H\a+k\e$ are precisely the values of the function $\rho_{\A(\a)}+k$ on the lines $h_1, \ldots, h_m$ of the arrangement $\A$. The lines of $\A(\a)$ are a subset of $\{h_1, \ldots, h_m\}$, so the ideal $I_{d,H\a+k\e}$ satisfies
\[
I'_{\A(\a),k} \subseteq I_{d,H\a+k\e} \subseteq I_{\A(\a),k}.
\]
For $k\in \{0,-1,-2\}$ we have $I'_{\A(\a),k}=I_{\A(\a),k}$ by Theorem \ref{th:I=I'}, so $I_{d,H\a+k\e}$ is equal to them also. It then follows that $I_{d,H\a+k\e}^{-1} \cong (C_{\A(\a), k})_d$ by Theorem \ref{th:iarrobino} and Corollary \ref{cor:JC}.
\end{proof}

When studying $\ZZ(\A, k\e)$ for general $k$, one runs into the same difficulties encountered in the study of the ideal $I'_{\A,k}$. When studying how polynomial functions on $V$ interact with a hyperplane arrangement $\A$ in $V$, it was somewhat unnatural to pay attention only to the lines of $\A$. Similarly, the zonotopal Cox ring of $\A$ pays attention almost exclusively to the lines of $\A$; the hyperplanes only play a role in the rank selection. It would be interesting to define a variant of the zonotopal Cox ring and modules which pays attention to the arrangement $\A$ in a more substantial way. It seems natural that this would involve the Cox ring of the wonderful compactification of $\A$ constructed by De Concini and Procesi \cite{DP1}.

Proposition \ref{prop:zonotopal} will allow us to compute the multigraded Hilbert series of an arbitrary zonotopal Cox ring, and of its central and interior Cox modules. We will do this in Section \ref{sec:zonotopalHilbert}, after a brief discussion on multivariate Tutte polynomials.


\subsection{Multivariate Tutte polynomials}

Let $\A$ be an arrangement of $n$ hyperplanes, and let $\v = (v_i)_{i \in \A}$ and $q$ be indeterminates. The \emph{multivariate Tutte polynomial} or  \emph{Potts model partition function} \cite{So} of $\A$ is
\[
{\widetilde Z}_{\A}(q; \v) = \sum_{\B \subseteq \A} q^{-r(\B)} \prod_{e \in \B} v_e.
\]
This is a polynomial in $q^{-1}$ and the $v_i$s. One can think of ${\widetilde Z}_{\A}(q; \v)$ as a multivariate Tutte polynomial where each hyperplane gets its own weight $v_e$; we obtain the ordinary Tutte polynomial when we give all hyperplanes the same weight:
\[
T_{\A}(x, y) = (x-1)^r {\widetilde Z}_{\A}((x-1)(y-1); y-1, y-1, \ldots, y-1).
\]
The polynomial ${\widetilde Z}_{\A}(q; \v)$ is defined in terms of the matroid $M(\A)$ only, and in turn it determines the matroid $M(\A)$ completely, since we can read the rank function from it.

The Tutte polynomials of $\A(\a)$ can also be computed from the multivariate Tutte polynomial of $\A$ as follows.

\begin{proposition} \label{prop:tuttemulti} If $\a=(a_1, \ldots, a_n) \in \mathbb{N}^n$, the Tutte polynomial of $\A(\a)$ is
\[
T_{\A(\a)}(x, y) = (x-1)^{r(\supp(\a))} {\widetilde Z}_{\A}((x-1)(y-1); y^{a_1}-1, y^{a_2}-1, \ldots, y^{a_n}-1).
\]
\end{proposition}

\begin{proof}
If an arrangement contains two copies $e$ and $f$ of the same hyperplane with weights $v_e$ and $v_f$, we can replace them by a single copy with weight $v_e+v_f+v_ev_f=(1+v_e)(1+v_f)-1$, and the resulting evaluation of the multivariate Tutte polynomial will not change. \cite{So} The Tutte polynomial of $\A(\a)$ is obtained by assigning a weight of $y-1$ to all elements of $\A(\a)$.
If $\A(\a)$ contains $a_i \geq 1$ copies of hyperplane $H_i$, we can merge these into a single copy having weight $(1+(t-1))^{a_i}-1$. If a hyperplane $H_i$ does not appear in $\A(\a)$ because $a_i=0$,  we can add a copy of it having weight $0=t^0-1$. In the end we are left with the arrangement $\A$ equipped with the desired weights.
\end{proof}
%



In Section \ref{sec:zonotopalHilbert} we will need to compute the weighted generating function for the Tutte polynomials of the arrangements $\A(\a)$, as $\a$ varies. The following technical lemma expresses that generating function in terms of the multivariate Tutte polynomials of $\A$ and its subarrangements.


\begin{lemma}\label{lemma:tutte}
If $\A$ is an arrangement of $n$ hyperplanes, and $q,x,y,w_1, \ldots, w_n$ are indeterminates, then
\begin{eqnarray*}
&& \sum_{\a \in \N^n} q^{r(\supp(\a))} \,T_{\A(\a)}(x,y) \,w_1^{a_1} \cdots w_n^{a_n}\\
&=&\sum_{\D \subseteq \A}(q(x-1))^{r(\D)} \prod_{i \in \D} \frac{w_i}{1-w_i} {\widetilde{Z}}_{\D}\left((x-1)(y-1),  \frac{y-1}{1-yw_1}, \ldots,  \frac{y-1}{1-yw_n}\right).
\end{eqnarray*}
\end{lemma}

\begin{proof}
By Proposition \ref{prop:tuttemulti}, the left hand side can be rewritten as
\begin{eqnarray*}
&&\sum_{\a \in \N^n} [q(x-1)]^{r(\supp(\a))}{\widetilde{Z}}_{\A}((x-1)(y-1); y^{a_1}-1, \ldots, y^{a_n}-1) w_1^{a_1} \cdots w_n^{a_n} \\ 
&=&\sum_{\D \subseteq \A} \sum_{\stackrel{\a \in \N^n}{\supp(\a)=\D}} \sum_{\B \subseteq \A} X^{r(\D)}Y^{-r(\B)}  \prod_{i \in \B} (y^{a_i}-1)w_i^{a_i}\prod_{i \notin \B} w_i^{a_i} \\
\end{eqnarray*}
where $X=q(x-1)$ and $Y=(x-1)(y-1)$. When $\D, \a$, and $\B$ are such that $\B$ is not contained in $\supp(\a)=\D$, there is an element $b \in \B$ with $a_b=0$ which makes the corresponding summand equal to $0$. Therefore our sum equals
\begin{eqnarray*}
&&\sum_{\D \subseteq \A}\sum_{\B \subseteq \D}  
\sum_{\stackrel{\a \in \N^n}{\supp(\a)=\D}}  X^{r(\D)}Y^{-r(\B)} \prod_{i \in \B} (y^{a_i}-1)w_i^{a_i}\prod_{i \in \D-\B} w_i^{a_i} \\
&=& \sum_{\D \subseteq \A}\sum_{\B \subseteq \D}   X^{r(\D)}Y^{-r(\B)} \prod_{i \in \B} \left(\sum_{a_i=1}^{\infty}(y^{a_i}-1)w_i^{a_i}\right) \prod_{i \in \D- \B} \left( \sum_{a_i=1}^{\infty}w_i^{a_i} \right)\\
&=& \sum_{\D \subseteq \A} \sum_{\B \subseteq \D} X^{r(\D)}Y^{-r(\B)}  \prod_{i \in \B} \left(\frac{yw_i}{1-yw_i} - \frac{w_i}{1-w_i}\right)
\prod_{i \in \D-\B} \left(\frac{w_i}{1-w_i}\right)\\
&=& \sum_{\D \subseteq \A}\left(X^{r(\D)} \prod_{i \in \D} \frac{w_i}{1-w_i} \sum_{\B \subseteq \D} \left( Y^{-r(\B)}  \prod_{i \in \B} \frac{y-1}{1-yw_i}\right)\right)
\end{eqnarray*}
which equals the right hand side.
\end{proof}

In Section \ref{sec:zonotopalHilbert} we will also need to compute three variants of the generating function of Lemma \ref{lemma:tutte}, two of which require special care. In the first variant, $q$ and $x$ lie on the hyperbola $q(x-1)=1$, and the right hand side can be expressed in terms of ${\widetilde Z}_{\A}$ only. In the second variant, we have $x=1$ and the right hand side is undefined.
%
We now treat those two cases.

\begin{lemma}\label{lemma:tutte1}
If $\A$ is an arrangement of $n$ hyperplanes, and $x,y,w_1, \ldots, w_n$ are indeterminates, then
\begin{eqnarray*}
&& \sum_{\a \in \N^n} (x-1)^{-r(\supp(\a))} \,T_{\A(\a)}(x,y) \,w_1^{a_1} \cdots w_n^{a_n}\\
&=&
\frac1{\prod_{i=1}^n(1-w_i)} {\widetilde{Z}_{\A}\left((x-1)(y-1); \frac{(y-1)w_1}{1-yw_1}, \cdots \frac{(y-1)w_n}{1-yw_n}\right)}.
%
\end{eqnarray*}
\end{lemma}

\begin{proof}
We use the notation of the proof of Lemma \ref{lemma:tutte}, where now $X=1$. The left hand side is
\begin{eqnarray*}
&& \sum_{\a \in \N^n} \sum_{\B \subseteq \A} Y^{-r(\B)}  \prod_{i \in \B} (y^{a_i}-1)w_i^{a_i}\prod_{i \notin \B} w_i^{a_i} \\
&=& \sum_{\B \subseteq \A} Y^{-r(\B)} \prod_{i \in \B} \left(\sum_{a_i=0}^{\infty}(y^{a_i}-1)w_i^{a_i}\right) \prod_{i \notin \B} \left( \sum_{a_i=0}^{\infty}w_i^{a_i} \right)
\end{eqnarray*}
where again $Y=(x-1)(y-1)$, and this equals the right hand side  by a similar argument.
\end{proof}

Notice that ${\widetilde{Z}}_{\A}(q; \v)$ is undefined at $q=0$. As $q \rightarrow 0$ we have
\[
\left.q^r \widetilde{Z}_{\A}(q; \v)\right|_{q=0} = S_{\A}(\v)
\]
where $S_{\A}(\v)$ is the generating polynomial for spanning sets:
\[
S_{\A}(\v) = \sum_{S \, : \, r(S) = r(\A)} \, \prod_{i \in S} v_i.
\]

\begin{lemma}\label{lemma:tutte2}
If $\A$ is an arrangement of $n$ hyperplanes, and $q,y,w_1, \ldots, w_n$ are indeterminates, then
\begin{eqnarray*}
&& \sum_{\a \in \N^n} q^{r(\supp(\a))} \,T_{\A(\a)}(1,y) \,w_1^{a_1} \cdots w_n^{a_n}\\
&=&\sum_{\D \subseteq \A}\left(\frac{q}{y-1}\right)^{r(\D)} \prod_{i \in \D} \frac{w_i}{1-w_i} S_{\D}\left(\frac{y-1}{1-yw_1}, \ldots,  \frac{y-1}{1-yw_n}\right).
\end{eqnarray*}
\end{lemma}

\begin{proof}
This follows by letting $(x-1)(y-1)=z$ and setting $z=0$ in Lemma \ref{lemma:tutte}.
\end{proof}

\subsection{Hilbert series}\label{sec:zonotopalHilbert}

\begin{theorem}\label{th:Coxring}
Let $\A=\{H_1, \ldots, H_n\}$ be a hyperplane arrangement, and let $h_1, \ldots, h_m$ be the lines in the arrangement. The multigraded Hilbert series of the zonotopal Cox ring $\ZZ(\A)$ is given by
\[
\Hilb(\ZZ(\A);t,s_1, \ldots, s_m) = 
\frac1{\prod_{i=1}^n(1-S_it)}
 \widetilde{Z}_{\A}\left(1-t; \frac{S_1(1-t)}{1-S_1}, \cdots \frac{S_n(1-t)}{1-S_n}\right)
\]
where $S_j = \prod_{i \, : \, h_i \notin H_j} s_i$ for $1 \leq j \leq n$.
\end{theorem}

\begin{proof}
Recall from Proposition \ref{prop:zonotopal} that $R^G_{(d, H\a)}$ is isomorphic to $(C_{\A(\a),0})_d$ as a vector space, and  has degree $de_0+(H\a)_1 \, e_1 + \cdots + (H\a)_m \, e_m$ in $\ZZ(\A)$, where 
$(H\a)_i = \sum_{j \, | \, h_i \notin H_j} a_j$. Thus
\begin{eqnarray*}
\Hilb(\ZZ(\A);t, s_1, \ldots, s_m) &=& \sum_{(d,\a) \in \mathbb{N}^{n+1}} \dim(C_{\A(\a),0})_d \,\, t^d\,  \prod_{i=1}^m s_i^{\sum_{j \, | \, h_i \notin H_j} a_j}\\
&=&  
\textrm{FakeHilb}(\ZZ(\A);t,S_1, \ldots, S_n),
\end{eqnarray*}
where
\[
\textrm{FakeHilb}(\ZZ(\A);t, t_1, \ldots, t_n) = \sum_{(d,\a) \in \mathbb{N}^{n+1}} \dim(C_{\A(\a),0})_d \,\, t^d\,  t_1^{a_1}\cdots t_n^{a_n}.
\]
Now, using the results of Corollary \ref{cor:hilbertC0} and Lemma \ref{lemma:tutte1}, we compute:
\begin{eqnarray*}
&&\textrm{FakeHilb}(\ZZ(\A);t, t_1, \ldots, t_n) = \sum_{\a \in \mathbb{N}^{n}} \Hilb(C_{\A(\a),0}; t) \,   t_1^{a_1}\cdots t_n^{a_n}\\
&=& \sum_{\a \in \mathbb{N}^{n}} t^{|\a|-r(\supp(\a))} T_{\A(\a)}\left(1+t, \frac1t\right) t_1^{a_1}\cdots t_n^{a_n}\\
&=& \sum_{\a \in \mathbb{N}^{n}} t^{-r(\supp(\a))} T_{\A(\a)}\left(1+t, \frac1t\right) (t_1t)^{a_1}\cdots (t_nt)^{a_n}\\
&=&\frac1{\prod_{i=1}^n(1-t_it)} {\widetilde{Z}_{\A}\left(1-t; \frac{t_1(1-t)}{1-t_1}, \cdots \frac{t_n(1-t)}{1-t_n}\right)},
\end{eqnarray*}
which gives the desired result.
\end{proof}

\begin{theorem}\label{th:Coxmodule}
In the notation of Theorem \ref{th:Coxring}, the multigraded Hilbert series of the central zonotopal Cox module is
\[
\sum_{\D \subseteq \A}(1-t)^{-r(\D)} \prod_{i \in \D} \frac{S_it}{1-S_it} S_{\D}\left(\frac{1-t}{t(1-S_1)}, \ldots,  \frac{1-t}{t(1-S_1)}\right)
\]
and the multigraded Hilbert series of the internal zonotopal Cox module is
\[
\sum_{\D \subseteq \A}\left(-\frac1t\right)^{r(\D)} \prod_{i \in \D} \frac{S_it}{1-S_it} {\widetilde{Z}}_{\D}\left(\frac{t-1}t,  \frac{1-t}{t(1-S_1)}, \ldots,  \frac{1-t}{t(1-S_1)}\right).
\]
\end{theorem}

\begin{proof}
Here $\Hilb(\ZZ(\A, -\e);t, s_1, \ldots, s_m) = \textrm{FakeHilb}(\ZZ(\A, -\e);t,S_1, \ldots, S_n)$,
which is 
\[
\sum_{\a \in \mathbb{N}^{n}} t^{-r(\supp(\a))} T_{\A(\a)}\left(1, \frac1t\right) (S_1t)^{a_1}\cdots (S_nt)^{a_n}
\]
and $\Hilb(\ZZ(\A, -2\e);t, s_1, \ldots, s_m) = \textrm{FakeHilb}(\ZZ(\A, -2\e);t,S_1, \ldots, S_n)$,
which is 
\[
\sum_{\a \in \mathbb{N}^{n}} t^{-r(\supp(\a))} T_{\A(\a)}\left(0, \frac1t\right) (S_1t)^{a_1}\cdots (S_nt)^{a_n}.
\]
It then remains to apply Lemmas \ref{lemma:tutte2} and \ref{lemma:tutte}, respectively.
\end{proof}

\section{Future directions.}

The following questions remain open.

\begin{itemize}
\item
Settle the various computational problems raised in Section \ref{sec:ideals}.
\item
What can we say about the space $C_{\A,k}$ for $k \leq -3?$
\item
What can we say about the space $C(\rho_{\A})$ for a \emph{subspace} arrangement $\A$?
\item
How does the Hilbert function of the space $C(\rho+k)$ depend on $k$ for an arbitrary proper function $\rho$? For the proper function $\rho_f$ of a polynomial? How does the Hilbert function of $J(\sigma+k)$ depend on $k$ for an arbitrary function $\sigma$? For a function $\sigma$ whose value at $h$ is the order of vanishing of a given function at $h$?
\item
Clarify the relationship between the zonotopal Cox ring of a hyperplane arrangement $\A$ and De Concini and Procesi's wonderful compactification of the complement of $\A$.

\end{itemize}

\section{Acknowledgments.} We would like to thank the anonymous referee for their careful reading of the manuscript and valuable suggestions.


\bibliographystyle{amsplain}

\end{document}